\documentclass[11pt]{amsart}
\usepackage[T1]{fontenc}
\usepackage{lmodern,microtype}
\usepackage[margin=1in]{geometry}
\usepackage{amsmath,amssymb,amsthm,mathrsfs}
\usepackage[colorlinks=true,linkcolor=blue,citecolor=blue,urlcolor=blue]{hyperref}
\numberwithin{equation}{section}
\newtheorem{theorem}{Theorem}[section]
\newtheorem{lemma}[theorem]{Lemma}
\newtheorem{proposition}[theorem]{Proposition}
\newtheorem{corollary}[theorem]{Corollary}
\theoremstyle{definition}
\newtheorem{definition}[theorem]{Definition}
\theoremstyle{remark}
\newtheorem{remark}[theorem]{Remark}
\newtheorem*{remark*}{Remark}
\newcommand{\PP}{\mathbb P}
\newcommand{\CC}{\mathbb C}
\newcommand{\ZZ}{\mathbb Z}
\newcommand{\OO}{\mathcal O}
\DeclareMathOperator{\Pic}{Pic}

\DeclareMathOperator{\RatCurves}{RatCurves}
\DeclareMathOperator{\Fitt}{Fitt}
\DeclareMathOperator{\Aut}{Aut}
\title[Homogeneity of contact Fano manifolds]
{Transverse rational curves and the homogeneity of contact Fano manifolds}
\author{Zhengyu Hu}
\address{Mathematical Sciences Research Center, Chongqing University of Technology,
Chongqing 400054, China}
\email{zhengyuhu16@gmail.com}

\author{Kefeng Liu}
\address{Mathematical Sciences Research Center, Chongqing University of Technology,
Chongqing 400054, China}
\email{kefengliu@cqut.edu.cn}

\author{Xueyuan Wan}
\address{Mathematical Sciences Research Center, Chongqing University of Technology,
Chongqing 400054, China}
\email{xwan@cqut.edu.cn}

\date{28 October 2026}
\hypersetup{
  pdftitle={Transverse rational curves and the homogeneity of contact Fano manifolds},
  pdfauthor={Zhengyu Hu, Kefeng Liu, Xueyuan Wan},
  pdfsubject={Version 2, 28 October 2026}
}
\begin{document}
\begin{abstract}
We prove that every smooth complex projective contact Fano manifold
is homogeneous under its connected contact automorphism group.
The argument extends and simplifies the incidence method of our earlier
proof of the LeBrun--Salamon conjecture. After the classical reductions,
balanced transverse rational maps through every point yield tangent
sections extending across the nodal boundary. These establish
irreducibility and nonlinearity of the general variety of minimal
rational tangents associated with contact lines, enabling Hwang's
descent theorem to produce global vector fields generating the tangent
bundle. Consequences include very ampleness of the contact line bundle,
embeddedness of every balanced rational map, K\"ahler--Einstein existence,
K-polystability, and the LeBrun--Salamon conjecture.
\end{abstract}
\maketitle

\section{Introduction}

Let $Y$ be a connected smooth complex projective variety of dimension
$2n+1$, with $n\ge0$. A contact structure is a subbundle $D\subset TY$ for which the
quotient $L=TY/D$ is a line bundle and the bracket pairing
$\omega:\wedge^2D\to L$, also called the Levi form, is nondegenerate.
For $n=0$, one has $D=0$ and the contact condition is vacuous.
The quotient map $\theta$ gives
\[
0\longrightarrow D\longrightarrow TY\xrightarrow{\theta}L
\longrightarrow0,\qquad -K_Y=(n+1)L.
\]
Thus $Y$ is Fano if and only if $L$ is ample. The contact Fano
classification problem asks whether every such variety is homogeneous
\cite{Beauville}. Its relation to positive quaternionic--K\"ahler
geometry is given by the twistor construction
\cite{Salamon82,LeBrun}: the twistor space of a compact positive
quaternionic--K\"ahler manifold is contact Fano and carries a
K\"ahler--Einstein metric. We prove the following statement for every
contact Fano manifold.

\begin{theorem}\label{cf:main}
Let $(Y,D)$ be a smooth complex projective contact Fano manifold.
Then the connected contact automorphism group
$\Aut(Y,D)^0$ acts transitively on $Y$. In particular, $TY$ is globally
generated and the contact line bundle $L$ is very ample.
\end{theorem}

Here $\Aut(Y,D)^0$ denotes the identity component of the algebraic
group of automorphisms of $Y$ preserving $D$.

In dimension at least three, the classical classification of
homogeneous contact manifolds \cite{Wolf} (see also
\cite[Theorem~0.1 and \S2]{Beauville}) makes Theorem~\ref{cf:main}
equivalent to the following statement: $(Y,D)$ is contact-isomorphic
to the adjoint variety $Y_{\mathfrak g}\subset\PP(\mathfrak g)$ of a
simple complex Lie algebra $\mathfrak g$. Here $Y_{\mathfrak g}$ is
the unique closed adjoint orbit, endowed with its canonical contact
structure. Under the contact isomorphism $\varphi:Y\to Y_{\mathfrak g}$,
the contact line bundle is identified as
\[
L\simeq\varphi^*\bigl(\OO_{\PP(\mathfrak g)}(1)|_{Y_{\mathfrak g}}\bigr).
\]
Conversely, the adjoint group acts transitively on $Y_{\mathfrak g}$
and preserves its canonical contact structure.

Our previous work \cite{HLW} proved the LeBrun--Salamon conjecture by
an incidence method developed independently of \cite{OpenAI}.
The present paper simplifies this independent approach and extends it
from twistor spaces to all contact Fano manifolds. We work directly
with tangent sections and infinitesimal preservation, applying Hwang's
descent theorem \cite{HwangDescent} in place of deformation rigidity
for families of maps. Twistor fibres provide embedded balanced
transverse rational curves through every point. For a general contact
Fano manifold, we first construct balanced transverse maps;
their embeddedness will follow from the descent argument below.

The preprint \cite[Theorem~1.1]{OpenAI} proves the same adjoint-variety
classification and also the LeBrun--Salamon conjecture. Its Section~8
gives the final identification directly; we invoke the classical
classification after proving contact homogeneity.

The common geometric input is the existence of a suitable pair of
contact lines and a first-order smoothing of their union
(\cite[Sections~3--4]{OpenAI}; see Subsections~\ref{cf:sec-lines}--
\ref{cf:sec-smoothing} below). We give this input in full, while
shortening the choice of the line pair by applying Buczy\'nski's secant
and join theorem \cite[Proposition~3.5]{Buczynski}.

Both arguments also extend the canonical two-point tensor, after
contraction with the contact differential, across the nodal boundary;
this construction is already present in \cite{HLW}. We use the
resulting sections on incidence varieties and apply Hwang's descent
theorem. We establish the required geometry before homogeneity: the
contact-line tangent-direction projection is smooth with irreducible
fibres, and its general fibre, the variety of minimal rational
tangents (VMRT), is nonlinear (Proposition~\ref{cf:nonlinear}).

The route in \cite[Sections~5--7]{OpenAI} descends the common contact
value of the extended tensors to a symmetric section of $L\boxtimes L$
on $Y\times Y$, and uses diagonal jets to obtain contact transitivity.

\hyperref[cf:proof-main]{The proof} is given in
Subsection~\ref{cf:sec-homogeneity}.

Applying Theorem~\ref{cf:main} to twistor spaces recovers the following
result of \cite{HLW}.

\begin{corollary}
\label{cf:wolf}
Let $(X^{4n},g)$ be a compact connected positive
quaternionic--K\"ahler manifold, with $n\ge2$. After multiplying $g$
by a positive constant, $(X,g)$ is isometric to a Wolf space.
\end{corollary}

Here a Wolf space is a compact simply connected irreducible
quaternionic--K\"ahler symmetric space with its standard invariant
metric \cite{Wolf,Salamon82}. The \hyperref[cf:proof-wolf]{proof below}
combines Theorem~\ref{cf:main} with the classification of homogeneous
contact manifolds and LeBrun's uniqueness theorem.

For arbitrary contact Fano manifolds, homogeneity also gives the
following metric and stability consequences.

\begin{corollary}
\label{cf:ke-polystability}
Under the hypotheses of Theorem~\ref{cf:main}, $Y$ is a rational
homogeneous variety. It admits a K\"ahler--Einstein metric of positive
Ricci curvature, and $(Y,-K_Y)$ is K-polystable.
\end{corollary}

The \hyperref[cf:proof-ke-polystability]{proof} combines homogeneity
with the existence of invariant
K\"ahler--Einstein metrics on rational homogeneous varieties
\cite{Matsushima72} and Berman's theorem \cite{Berman16}.

For $\dim Y=2n+1$, we call a map $f:\PP^1\to Y$ \emph{balanced} if
\[
f^*TY\simeq\OO(2)\oplus\OO(1)^{\oplus2n}.
\]
The next corollary combines embeddedness of balanced maps with the
existence of a twistor family, supplied by LeBrun's correspondence
\cite[Theorem~A]{LeBrun}.

\begin{corollary}\label{cf:embedded-bt}
Let $(Y,D)$ be as in Theorem~\ref{cf:main}, with $\dim Y=2n+1$.
\begin{enumerate}
\renewcommand{\labelenumi}{\textup{(\arabic{enumi})}}
\item Every balanced map $f:\PP^1\to Y$ is a closed immersion, with
$N_{f(\PP^1)/Y}\simeq\OO(1)^{\oplus2n}$.
Every point of $Y$ lies on such a curve $C$ satisfying
$\theta|_{TC}:TC\xrightarrow{\sim}L|_C$.
\item For $n\ge1$, these transverse curves may be chosen as the fibres
of a $C^\infty$ twistor submersion onto a positive
quaternionic--K\"ahler $4n$-manifold.
\end{enumerate}
\end{corollary}

In real dimension four, quaternionic--K\"ahler means self-dual Einstein
\cite[Definition~2.3]{LeBrun}.

An embedded smooth rational curve with this normal bundle is called a
\emph{quasi-line}. Proposition~\ref{cf:bt-embedding} proves embeddedness
directly by incidence and descent for transverse maps in the primitive
Picard case. After homogeneity, Proposition~\ref{cf:degree-two-images}
gives the assertion for every balanced map, without transversality.
The \hyperref[cf:proof-embedded-bt]{proof of the corollary} is given at
the end of Subsection~\ref{cf:sec-homogeneity}.

\begin{remark}\label{cf:weak-fano}
The ampleness of the contact line bundle in Theorem~\ref{cf:main}
cannot be replaced by nefness and bigness. Let $S=\PP^1\times\PP^1$
and $Y=\PP(T^*S)$, where projectivization parametrizes lines.
This is a smooth projective contact threefold with contact line bundle
$L=\OO_Y(1)$. Since $TS=\OO(2,0)\oplus\OO(0,2)$ is globally
generated, so is $L$, and the projective-bundle formula gives
\[
L^3=\int_S\bigl(c_1(TS)^2-c_2(TS)\bigr)=8-4=4.
\]
Thus $L$ is nef and big. The summand $\OO(-2,0)\subset T^*S$
defines a section $s:S\to Y$ with $s^*L=\OO(2,0)$, so
$L\cdot C=0$ for $C=s(\{p\}\times\PP^1)$, with $p\in\PP^1$. Consequently
$-K_Y=2L$ is nef and big but not ample: $Y$ is weak Fano.
The anticanonical morphism is generically finite but contracts $C$.
It is $\Aut(Y)^0$-equivariant, so transitivity would force all its
fibres to have the same dimension. Hence $Y$ is not homogeneous.
Here $\rho(Y)=3$; in Picard number one, a nef and big line bundle
is already ample.
\end{remark}

\noindent\textbf{Proof strategy and organization.}
The classical reductions of Lemma~\ref{cf:reduction} leave
$\dim Y=2n+1\ge5$ and $\Pic(Y)=\ZZ[L]$.
Two kinds of rational curves play different roles:
maps of $L$-degree two produce tangent sections, while contact lines
of $L$-degree one supply the varieties of minimal rational tangents
(VMRTs) used to descend those sections.

Section~\ref{cf:sec-curves} starts with two contact lines through a
general point $y$, chosen with complementary positive tangent
subspaces and tangent vectors $a,b$ with $\omega_y(a,b)\ne0$.
Write $s$ for the smoothing parameter and $\operatorname{Ev}$ for
two-point evaluation. With the nodal source fixed, its differential
has cokernel $L_y$ (Lemma~\ref{cf:evaluation-cokernel}). The induced
map from the smoothing direction to this cokernel satisfies
\[
\overline{d\operatorname{Ev}}(\partial_s)=\omega_y(a,b)
\]
(Lemma~\ref{cf:fixed-end-obstruction}). Its nonvanishing supplies the
missing evaluation direction and yields a balanced map
(Proposition~\ref{cf:first-conic}).

The two-line boundary estimate then verifies the general propagation
criterion of Proposition~\ref{cf:very-free-propagation}. First-jet
deformations give balanced transverse maps through every point, with
tangent lines spanning each tangent space of $Y$
(Corollary~\ref{cf:every-point}).

Section~\ref{cf:sec-tangent-sections} uses these maps to construct
tangent sections and prove homogeneity. At a nodal fibre, the relative
contact differential cancels the pole of the two-point tensor
(Lemma~\ref{cf:cancellation}), extending the incidence sections to
projective sources (Theorem~\ref{cf:incidence}). On the \'etale locus
of each evaluation morphism, its section becomes a local vector field.
These fields preserve the total space of contact-line tangent
directions by Proposition~\ref{cf:jets}. Their values span each
tangent space, so the smoothness criterion of
Lemma~\ref{cf:smoothness} makes the tangent-direction projection
smooth. Stein factorization and the Fano condition give irreducible
fibres; the dimension and span of the general VMRT then imply
nonlinearity (Proposition~\ref{cf:nonlinear}).

With these hypotheses established, Hwang's theorem descends the same
incidence sections to global vector fields. The general criterion of
Proposition~\ref{cf:generation} gives global generation of $TY$ and
homogeneity, and uniqueness of the contact structure gives contact
homogeneity. The metric and twistor consequences follow afterwards.
The propagation, smoothness and homogeneity criteria are formulated
independently of the contact structure.

\noindent\textbf{AI declaration.}  ChatGPT was used solely as an auxiliary tool for literature searches, language editing, and preliminary checks of mathematical constructions and arguments. The authors independently verified all proofs and mathematical claims and take full responsibility for the content of this paper.

\noindent\textbf{Conventions.}
All varieties, schemes and stacks are over $\CC$; a variety means an
integral separated scheme of finite type, unless stated otherwise.
Irreducible components carry their reduced structures; Hom schemes
and scheme-theoretic fibres retain their natural scheme structures.
We use $\PP(E)$
for one-dimensional subspaces of the fibres of $E$. For a morphism
$f:X\to Y$ between smooth varieties, $df:TX\to f^*TY$ denotes its
differential as a homomorphism of vector bundles. The pullback of the
contact form is denoted by $f^*\theta:f^*TY\to f^*L$.

\subsection*{Classical reductions}

\begin{lemma}\label{cf:reduction}
Let $(Y,D)$ be a smooth complex projective contact Fano manifold of
dimension $2n+1\ge3$, with contact line bundle $L$.
Either $(Y,D)$ is contact-isomorphic to projective space
$\PP^{2n+1}$ with its standard contact structure or to
$\PP(T^*\PP^{n+1})$ with its tautological contact structure, or
$n\ge2$ and $\operatorname{Pic}(Y)=\ZZ[L]$.
The two listed models are homogeneous under connected groups of
contact automorphisms.
\end{lemma}

\begin{proof}
Kodaira vanishing gives $H^i(Y,\OO_Y)=0$ for $i>0$.
For a torsion line bundle $A$, the bundle $A\otimes K_Y^{-1}$ is ample,
so Kodaira vanishing also gives $H^i(Y,A)=0$ for $i>0$.
Riemann--Roch gives $\chi(A)=\chi(\OO_Y)=1$.
For a nonzero section $s$ of $A$ and $m>0$ with
$A^{\otimes m}\simeq\OO_Y$, the section $s^{\otimes m}$ is a nonzero
constant. Hence $s$ is nowhere zero and $A$ is trivial.
Thus $\operatorname{Pic}(Y)$ is free of rank $b_2(Y)$.

If $b_2(Y)>1$, \cite[Theorem~1.1]{KPSW} gives
$(Y,L)\simeq(\PP(T^*B),\OO(1))$ for a smooth projective
$(n+1)$-fold $B$; here projectivization parametrizes lines, whereas
\cite{KPSW} uses quotients. Since $L$ is ample, $TB$ is ample,
and Mori's theorem \cite{Mori} gives $B\simeq\PP^{n+1}$.
The description of contact forms in \cite[Proposition~2.14]{KPSW}
shows that the given contact structure is carried to the
tautological one by the projectivization of a bundle automorphism
of $T^*B$. The natural action of $\operatorname{PGL}_{n+2}$
on the resulting incidence variety is transitive and preserves
this contact structure.

Suppose now that $b_2(Y)=1$, and write
$\operatorname{Pic}(Y)=\ZZ[A]$, with $A$ ample, and $L=A^{\otimes k}$.
If $k\ge2$, the index $(n+1)k$ is at least $2n+2$.
The Kobayashi--Ochiai index theorem \cite{KO} then gives
$k=2$ and $Y\simeq\PP^{2n+1}$, with $L\simeq\OO(2)$.
If $n=1$, the same conclusion $Y\simeq\PP^3$ follows from
\cite[Theorem~1.2]{BuczynskiPeternell}, and again $L\simeq\OO(2)$.
On $\PP(V)$, the Euler sequence identifies
$H^0(\Omega^1_{\PP(V)}(2))$ with $\wedge^2 V^*$.
A contact form corresponds to a nondegenerate alternating form;
all such forms are linearly equivalent. Hence the given structure
is standard and its symplectic group acts transitively.
The remaining case is $k=1$ and $n\ge2$.
\end{proof}

\section{Transverse rational curves}\label{cf:sec-curves}

Throughout this section,
let $(Y,D)$ be a smooth complex
contact Fano manifold with $\dim Y=2n+1$, $n\ge2$, and
$\operatorname{Pic}(Y)=\ZZ[L]$. Write
\[
0\longrightarrow D\longrightarrow TY\xrightarrow{\theta}L
\longrightarrow0,\qquad -K_Y=(n+1)L.
\]
\subsection{Basic definitions and deformation facts}\label{cf:preliminaries}

The \emph{contact degree} of $f:\PP^1\to Y$ is $\deg f^*L$.
Since $\det TY\simeq-K_Y\simeq L^{\otimes(n+1)}$, one has
\begin{equation}\label{cf:degree-relation}
\deg f^*TY=(n+1)\deg f^*L.
\end{equation}
Such a map is \emph{free} if $f^*TY$ is globally generated, and
\emph{very free} if all its splitting summands have positive degree.
It is \emph{unramified} if $df$ is nonzero everywhere; this does not
assert injectivity on points. It is \emph{horizontal} if
$(f^*\theta)\circ df=0$, and \emph{transverse} at a point if this
contact differential is nonzero there.
A degree-one map is called a \emph{contact line}. It is birational
onto its image and horizontal, since
$\operatorname{Hom}(\OO(2),\OO(1))=0$.

\begin{lemma}\label{cf:degree-two-transversality}
A morphism $f:\PP^1\to Y$ of contact degree two is either everywhere
horizontal or everywhere transverse to $D$. In particular, transversality
at one point implies transversality everywhere.
\end{lemma}

\begin{proof}
Since $f^*L\simeq\OO(2)$, the contact differential
\[
(f^*\theta)\circ df:T\PP^1\simeq\OO(2)\longrightarrow f^*L\simeq\OO(2)
\]
is either zero or an isomorphism.
\end{proof}

\begin{definition}\label{cf:balanced-definition}
A map $f:\PP^1\to Y$ is \emph{balanced} if
\[
f^*TY\simeq\OO(2)\oplus\OO(1)^{\oplus2n}.
\]
It is \emph{balanced transverse} if, in addition, the contact differential
$(f^*\theta)\circ df:T\PP^1\to f^*L$ is an isomorphism.
\end{definition}

A balanced map is very free and has contact degree two by
\eqref{cf:degree-relation}. Conversely, a very free map of contact
degree two is balanced: its $2n+1$ positive splitting degrees sum
to $2n+2$.
Since $\operatorname{Hom}(\OO(2),\OO(1))=0$, the nonzero differential of
a balanced map
identifies $T\PP^1\simeq\OO(2)$ with the degree-two summand of $f^*TY$.
Thus every balanced map is unramified, with normal bundle
$N_f=\operatorname{coker}(df)\simeq\OO(1)^{\oplus2n}$.
By Lemma~\ref{cf:degree-two-transversality}, transversality at one point
makes it balanced transverse.

\begin{lemma}\label{cf:deformation-facts}
Let $X$ be a smooth projective variety, and let $f:\PP^1\to X$.
\begin{enumerate}
\renewcommand{\labelenumi}{\textup{(\arabic{enumi})}}
\item\label{cf:hom-smoothness}
The morphism scheme $\underline{\operatorname{Hom}}(\PP^1,X)$ is separated
and locally of finite type. It has tangent space $H^0(\PP^1,f^*TX)$
at $[f]$ and is smooth there if
$H^1(\PP^1,f^*TX)=0$.
\item\label{cf:evaluation-smoothness}
Evaluation at distinct fixed points $p_1,\ldots,p_r$ is smooth at
$[f]$ if $H^1(f^*TX(-\sum_i[p_i]))=0$; its differential is
evaluation of sections. The scheme-theoretic fibre through $[f]$
has tangent space $H^0(f^*TX(-\sum_i[p_i]))$.
\item\label{cf:jet-smoothness}
Fix $p\in\PP^1$ and a local coordinate $u$ at $p$. The first-jet
evaluation is
\[
j_p^1:\underline{\operatorname{Hom}}(\PP^1,X)\longrightarrow TX,
\qquad [h]\longmapsto\bigl(h(p),dh_p(\partial_u)\bigr).
\]
Here $TX$ denotes the total space of the tangent bundle.
If $H^1(f^*TX(-2[p]))=0$, this map is smooth at $[f]$.
For $x=f(p)$, its restriction
\[
\underline{\operatorname{Hom}}(\PP^1,X;p\mapsto x)\longrightarrow T_xX,
\qquad [h]\longmapsto dh_p(\partial_u)
\]
is also smooth at $[f]$; the pointed Hom scheme is the scheme-theoretic
fibre of evaluation at $p$ over $x$.
\item\label{cf:degree-constancy}
In a projective flat family of curves, the fibre degree of a line
bundle is locally constant.
\end{enumerate}
\end{lemma}

\begin{proof}
The first three assertions are standard deformation theory
\cite[I.2.16--I.2.17]{Kollar}, with sections vanishing at the prescribed
points or to order two at $p$. The last follows from constancy of
Euler characteristic and $\deg M=\chi(M)-\chi(\OO_C)$ on a fibre $C$.
\end{proof}

For a projective variety $X$ and a curve class $\beta\in H_2(X,\ZZ)$,
write $\underline{\operatorname{Hom}}_\beta(\PP^1,X)$ for the locus
of morphisms $f$ with $f_*[\PP^1]=\beta$, with its induced scheme structure.
The homology class is locally constant in families, and fixing $\beta$
fixes the degree for a projective embedding of $X$. Thus this is an
open-and-closed quasi-projective subscheme of finite type.
In the contact setting, we
write $\underline{\operatorname{Hom}}_2(\PP^1,Y)$ for the locus of maps
of contact degree two, also a quasi-projective scheme of finite type.
It is open and closed by
Lemma~\ref{cf:deformation-facts}\textup{(\ref{cf:degree-constancy})}.

We write $\overline{\mathcal M}_{0,k}(X,\beta)$ for the proper
Deligne--Mumford stack of stable maps and $\overline M_{0,k}(X,\beta)$
for its projective coarse moduli scheme \cite{BM,FP}.

Assume now that $X$ is smooth. For a map $f:\PP^1\to X$ of class
$\beta$ birational onto its image,
the stable-map automorphism group is trivial. Let $r=1$ or $2$, take
distinct source points $p_1,\ldots,p_r$, and put
$P=\sum_i[p_i]$ and $\mathbf x=(f(p_1),\ldots,f(p_r))$.
We may fix the source points by reparametrization. By
Lemma~\ref{cf:deformation-facts}\textup{(\ref{cf:evaluation-smoothness})},
deformations fixing their images are sections of $f^*TX(-P)$, and reparametrizations
preserving the source points contribute $df\,H^0(T\PP^1(-P))$.
For $\mathrm{ev}:\overline M_{0,r}(X,\beta)\to X^r$,
\begin{equation}\label{cf:pointed-moduli-tangent}
T_{[f,p_1,\ldots,p_r]}\mathrm{ev}^{-1}(\mathbf x)
\simeq
\frac{H^0(f^*TX(-P))}{df\,H^0(T\PP^1(-P))}.
\end{equation}
Its dimension is $h^0(f^*TX(-P))-(3-r)$, computed from the splitting
of $f^*TX$ using $h^0(\OO(d))=\max\{d+1,0\}$.
If $H^1(f^*TX(-P))=0$, the evaluation fibre is smooth there, so this
is also its local dimension. Only the source points must be distinct;
their images may coincide.

\subsection{A pair of contact lines}\label{cf:sec-lines}

We adapt the line-pair and smoothing construction of
\cite[Sections~3--4]{OpenAI}. The secant argument below replaces its
cone-addition calculation by \cite[Proposition~3.5]{Buczynski}.

\begin{definition}[{Minimal rational component, \cite[Section~2]{HwangDescent}}]\label{cf:mrc-definition}
An irreducible component $\mathcal H\subset\RatCurves^n(Y)$ is called a
\emph{minimal rational component} if, for a general point $y\in Y$,
the subfamily $\mathcal H_y$ parametrizing its members through $y$,
with its reduced structure, is nonempty and proper.
\end{definition}

Here $\RatCurves^n(Y)$ denotes the normalized space of rational curves.
Choose a covering irreducible component $\mathcal K$ of contact lines
in this space, as provided by \cite[Section~2.3]{Kebekus}.
It is a normal quasi-projective variety, and is proper by
\cite[Remark~2.2]{Kebekus}, hence projective.
Its normalized universal family is projective, and its evaluation
map is proper and surjective.
Thus, for every $y\in Y$, the subfamily $\mathcal K_y$ is nonempty and,
being closed in the projective variety $\mathcal K$, is projective.
Hence $\mathcal K$ is a minimal rational component.

\begin{definition}[{Variety of minimal rational tangents, \cite[Section~2]{HwangDescent}}]\label{cf:vmrt-definition}
For $y\in Y$, the \emph{variety of minimal rational tangents} (VMRT)
associated with $\mathcal K$ is the reduced projective closed subscheme
\[
Z_y:=\overline{\bigl\{[T_yC]\mid [C]\in\mathcal K,\ y\in C,
\ C\text{ smooth at }y\bigr\}}_{\mathrm{red}}
\subset\PP(D_y).
\]
Here $[T_yC]$ denotes the tangent line of $C$ at $y$; the term VMRT
allows $Z_y$ to be reducible.
\end{definition}

\begin{theorem}[Contact lines and their VMRT]\label{cf:vmrt-facts}
For a general point $y\in Y$, the following hold.
\begin{enumerate}
\renewcommand{\labelenumi}{\textup{(\arabic{enumi})}}
\item\label{cf:vmrt-lines}
Every member of $\mathcal K$ through $y$ is smooth and its
parametrization is free \cite[Proposition~3.3 and Lemma~3.5]{Kebekus}.
\item\label{cf:vmrt-geometry}
$Z_y$ is smooth and Legendrian in $\PP(D_y)$, of dimension $n-1$
\cite[Theorem~3.1(ii)--(iii)]{Buczynski}.
\item\label{cf:vmrt-span}
The linear span of $Z_y$ is $\PP(D_y)$
\cite[Theorem~4.4]{Kebekus}.
\end{enumerate}
\end{theorem}

For general $y$, let $\widehat Z_y\subset D_y$ be the reduced affine cone
over $Z_y$. The Legendrian condition means that, for each
$0\ne v\in\widehat Z_y$, the tangent space $T_v\widehat Z_y$ is a
Lagrangian $n$-plane for the Levi pairing
$\omega_y:\wedge^2D_y\to L_y$.

\begin{remark*}
These properties will be used to construct a suitable pair of contact
lines below, and again in Subsection~\ref{cf:sec-homogeneity} to study
the VMRT projection and prove nonlinearity. The theorem does not
assert irreducibility of $Z_y$; for general $y$, this will be proved in
Proposition~\ref{cf:nonlinear}.
\end{remark*}

\begin{remark*}
For completeness, we record two general facts about minimal rational
curves and their VMRTs. Neither requires a contact structure.
\begin{enumerate}
\renewcommand{\labelenumi}{\textup{(\arabic{enumi})}}
\item \emph{Projective-bundle criterion.} Let
$\pi:\mathcal W\to B$ be a proper smooth morphism of finite-type
schemes whose geometric fibres are $\PP^1$. If a line bundle $A$ on
$\mathcal W$ restricts to $\OO_{\PP^1}(1)$ on every geometric fibre,
then cohomology and base change give a rank-two vector bundle
$E=\pi_*A$, and evaluation identifies
\[
\mathcal W\simeq\PP_B(E^\vee).
\]
Thus $\pi$ is a Zariski locally trivial $\PP^1$-bundle. For the
normalized universal family of contact lines above, the evaluation
pullback of $L$ supplies $A$.
\item \emph{Normalization via the tangent map.} Let $\mathcal H$ be a
minimal rational component on a smooth projective variety $X$, with
normalized universal family $\rho:\mathcal U_{\mathcal H}\to\mathcal H$
and evaluation $e:\mathcal U_{\mathcal H}\to X$. For general $x\in X$,
the scheme $e^{-1}(x)$ is smooth and projective, and the tangent map
\[
\tau_x:e^{-1}(x)\longrightarrow Z_x
\]
is finite and birational, where $Z_x$ is the VMRT of $\mathcal H$ at
$x$ \cite[Propositions~3.7,~4.3 and Theorem~4.18]{HwangVMRT}.
Thus $\tau_x$ is the normalization map; in particular, it is an
isomorphism if $Z_x$ is normal.
\end{enumerate}
\end{remark*}

\begin{lemma}\label{cf:positive-part}
\leavevmode
\begin{enumerate}
\renewcommand{\labelenumi}{\textup{(\arabic{enumi})}}
\item\label{cf:positive-part-splitting}
Every free contact line $f:\PP^1\to Y$ has the splitting
\begin{equation}\label{cf:line-splitting}
f^*TY\simeq\OO(2)\oplus\OO(1)^{\oplus(n-1)}
\oplus\OO^{\oplus(n+1)}.
\end{equation}
The sum $P_f$ of its positive-degree summands is independent of
the splitting and is a Lagrangian subbundle of $f^*D$.
\item\label{cf:positive-part-tangent}
For general $y\in Y$ and every $a\in\widehat Z_y\setminus\{0\}$,
there is a free parametrization $f:\PP^1\to Y$ of a smooth member
of $\mathcal K$ with
\[
f(0)=y,\qquad f'(0)=a,\qquad (P_f)_0=T_a\widehat Z_y.
\]
Here the derivative is taken in a fixed local parameter at $0$;
the last equality uses the natural identifications
$(f^*D)_0\simeq D_y$ and $T_aD_y\simeq D_y$.
\end{enumerate}
\end{lemma}

\begin{proof}
For \textup{(\ref{cf:positive-part-splitting})}, write $T'=f^*TY$
and $D'=f^*D$, with $f^*L\simeq\OO(1)$.
We first prove \eqref{cf:line-splitting} (see also
\cite[Lemma~3.5]{Kebekus}). Freeness gives a splitting
$T'=P_f\oplus\OO^{\oplus r}$ into positive and trivial parts.
Since $\operatorname{Hom}(\OO(2),\OO(1))=0$, one has
$df:\OO(2)\to D'$. This map is nonzero because $f$ is birational
onto its image. The nondegenerate Levi form $\omega$ and the dual of $df$
give a nonzero composite
\[
D'\xrightarrow[\sim]{(f^*\omega)^\flat}
(D')^\vee\otimes f^*L
\xrightarrow{(df)^\vee\otimes\mathrm{id}_{f^*L}}
\OO(2)^\vee\otimes f^*L\simeq\OO(-1).
\]
Thus $D'$ is not globally generated.
If $f^*\theta$ were nonzero on a positive summand $\OO(a)$,
then $a=1$ and its restriction would be an isomorphism onto $f^*L$.
This would split the pulled-back contact sequence and make $D'$
globally generated, a contradiction. Hence $P_f\subset D'$.
The induced surjection $\OO^{\oplus r}\to\OO(1)$ has kernel
$\OO(-1)\oplus\OO^{\oplus(r-2)}$ by the evaluation sequence, so
\[
D'\simeq P_f\oplus\OO(-1)\oplus\OO^{\oplus(r-2)}.
\]
Since $D'\simeq(D')^\vee\otimes\OO(1)$, the splitting degrees
are paired by $a\leftrightarrow1-a$. The unique $\OO(-1)$ therefore
pairs with a unique $\OO(2)$, and each trivial summand with an
$\OO(1)$. Consequently
\[
P_f\simeq\OO(2)\oplus\OO(1)^{\oplus(r-2)},\qquad
2n=\operatorname{rank}D'=2r-2.
\]
Thus $r=n+1$, proving \eqref{cf:line-splitting}.

For any other splitting $T'=P'\oplus Q'$ into positive and trivial
parts, the projection $P_f\to Q'$ is zero. Hence $P_f\subseteq P'$,
and symmetry gives equality, proving independence of the splitting.
The restriction of $f^*\omega$ to $\wedge^2P_f$ vanishes, since
$\operatorname{Hom}(\OO(a+b),\OO(1))=0$ for $a,b\ge1$.
Together with $\operatorname{rank}P_f=n$, this makes its fibres
Lagrangian.

For \textup{(\ref{cf:positive-part-tangent})}, fix $y,a$ as above and a
local parameter $u$ at $0$. The pointed members of $\mathcal K$ through
$y$ form a projective parameter scheme. Their tangent-direction map
has closed image $Z_y$, so we may choose a free parametrization $f$
with $f(0)=y$ and $f'(0)=a$. Put $T'=f^*TY$.
The pointed Hom scheme is smooth at $f$ since $H^1(T'(-[0]))=0$
(Lemma~\ref{cf:deformation-facts}\textup{(\ref{cf:evaluation-smoothness})}).
On a smooth quasi-projective neighbourhood $\mathcal H_y$ whose maps
parametrize members of $\mathcal K$, differentiation at $0$ defines
\[
\tau:\mathcal H_y\longrightarrow\widehat Z_y\setminus\{0\},
\qquad g\longmapsto g'(0),
\]
with differential
\[
d\tau_f:H^0(T'(-[0]))\longrightarrow T_a\widehat Z_y,
\qquad s\longmapsto s'(0).
\]
Since $s(0)=0$, $s'(0)\in T_yY$ is independent of target coordinates.
By \eqref{cf:line-splitting}, the trivial summands contribute no
sections, while the positive summands allow arbitrary first derivatives.
Thus
\[
(P_f)_0=\operatorname{Im}(d\tau_f)\subseteq T_a\widehat Z_y.
\]
The target has dimension $n$ by
Theorem~\ref{cf:vmrt-facts}\textup{(\ref{cf:vmrt-geometry})}, and $P_f$
has rank $n$, proving equality. Together with
\textup{(\ref{cf:positive-part-splitting})}, this also recovers the
Legendrian property of $Z_y$.
\end{proof}

\begin{lemma}\label{cf:line-pair}
For a general point $y\in Y$, there are free contact lines
$f_1,f_2:\PP^1\to Y$ in $\mathcal K$ with
$f_1(0)=f_2(0)=y$ such that, writing
$a=f'_1(0)$ and $b=f'_2(0)$, one has
\begin{equation}\label{cf:pair-conditions}
(P_{f_1})_0\oplus(P_{f_2})_0=D_y,
\qquad \omega_y(a,b)\ne0.
\end{equation}
\end{lemma}

\begin{proof}
If $Z_y$ is irreducible, set $Z_1=Z_2=Z_y$; its secant variety
fills $\PP(D_y)$ by \cite[Proposition~3.5]{Buczynski}.
Otherwise, choose two distinct components $Z_1,Z_2$. They are
disjoint because $Z_y$ is smooth, and their join fills $\PP(D_y)$
by the same result.
Here the join is the Zariski closure of the union of lines joining
points on the two components. Writing $\widehat Z_i\subset D_y$
for the affine cone over $Z_i$, the addition map
\[
\sigma:\widehat Z_1\times\widehat Z_2\longrightarrow D_y,
\qquad (a,b)\longmapsto a+b
\]
is therefore dominant. By generic smoothness in characteristic zero,
its differential $d\sigma_{(a,b)}(u,w)=u+w$ is surjective at a general
pair of nonzero points. Thus $T_a\widehat Z_1+T_b\widehat Z_2=D_y$;
since both tangent spaces have dimension $n$, they are complementary.

Such a pair can also be chosen with nonzero Levi pairing.
Indeed, if $\omega_y$ vanished identically on
$\widehat Z_1\times\widehat Z_2$, differentiating in the second variable
at a pair with complementary tangent spaces would give
$\omega_y(a,T_b\widehat Z_2)=0$. Since $T_b\widehat Z_2$ is Lagrangian,
this implies $a\in T_b\widehat Z_2$. But $0\ne a\in T_a\widehat Z_1$
because $\widehat Z_1$ is a cone, contradicting complementarity.

Thus $\omega_y(a,b)\ne0$ defines a nonempty open subset of the
irreducible product $\widehat Z_1\times\widehat Z_2$, which meets
the open set of pairs with complementary tangent spaces.
Applying Lemma~\ref{cf:positive-part}\textup{(\ref{cf:positive-part-tangent})}
to a general pair in this intersection gives free contact lines satisfying
\eqref{cf:pair-conditions}.
\end{proof}

\subsection{Smoothing to a very free conic}\label{cf:sec-smoothing}

Fix $f_1,f_2$ as in Lemma~\ref{cf:line-pair}, with
$f_1(0)=f_2(0)=y$, and consider the projective flat family
\begin{equation}\label{cf:source-smoothing}
\mathscr C=\{U_1V_1=sU_0V_0\}
\subset\PP^1\times\PP^1\times\mathbb A^1_s
\longrightarrow\mathbb A^1_s.
\end{equation}
Its total space is a smooth quasi-projective surface, with central
fibre $\mathscr C_0=C_1\cup C_2$, where $C_1=\PP^1\times\{0\}$ and
$C_2=\{0\}\times\PP^1$ meet at $o=(0,0)$.
In coordinates $u=U_1/U_0$ and $v=V_1/V_0$, the equation is $uv=s$.
The marked sections are $p(s)=(\infty,0,s)$ and $q(s)=(0,\infty,s)$.
Let $f_0:\mathscr C_0\to Y$ restrict to $f_i$ on $C_i$, and put
$T_0=f_0^*TY$ and $T_i=T_0|_{C_i}=f_i^*TY$.

Let
\[
\sigma:\mathscr H:=
\underline{\operatorname{Hom}}_{\mathbb A^1}(\mathscr C,Y\times\mathbb A^1)
\longrightarrow\mathbb A^1_s
\]
be the relative Hom scheme. We work in a quasi-projective open
neighbourhood of $[f_0]$. Evaluation at $p,q$ defines
\[
\operatorname{Ev}:\mathscr H\longrightarrow Y\times Y,
\qquad [f:\mathscr C_s\to Y]\longmapsto(f(p(s)),f(q(s))).
\]

Freeness and the normalization sequence give $H^1(\mathscr C_0,T_0)=0$,
so $\sigma$ is smooth at $[f_0]$ by relative deformation theory
\cite[I.2.16--I.2.17]{Kollar}. Thus $\mathscr H$ is smooth there, with
tangent sequence
\[
0\longrightarrow H^0(\mathscr C_0,T_0)
\longrightarrow T_{[f_0]}\mathscr H
\xrightarrow{d\sigma}T_0\mathbb A^1\longrightarrow0.
\]
We take $d\sigma$ and $d\operatorname{Ev}$ at $[f_0]$.
The kernel parametrizes deformations with the source fixed, and
$d\operatorname{Ev}$ restricts to evaluation at $p(0),q(0)$.

\begin{lemma}\label{cf:evaluation-cokernel}
Let $X$ be a smooth projective variety and let $f:C=C_1\cup C_2\to X$
be a morphism, where $C_1,C_2\simeq\PP^1$ meet at a single node $o$.
Suppose that $f_i=f|_{C_i}$ is free for $i=1,2$, and choose
$p\in C_1\setminus\{o\}$ and $q\in C_2\setminus\{o\}$.
Put $y=f(o)$, $T'=f^*TX$, and $T_i'=T'|_{C_i}=f_i^*TX$.
Let $P_i\subset T_i'$ be the sum of its positive-degree summands. Then
\[
\begin{aligned}
\operatorname{coker}\bigl(H^0(C,T')\to T'_p\oplus T'_q\bigr)
&\simeq H^1(C,T'(-[p]-[q]))\\
&\simeq \frac{T_yX}{(P_1)_o+(P_2)_o}.
\end{aligned}
\]
\end{lemma}

\begin{proof}
Freeness and the normalization sequence give $H^1(C,T')=0$.
The evaluation sequence then gives the first isomorphism.
The normalization sequence identifies $H^1(C,T'(-[p]-[q]))$ with the
cokernel of
\[
H^0(C_1,T_1'(-[p]))\oplus H^0(C_2,T_2'(-[q]))
\longrightarrow T_yX.
\]
This map is the difference of the node values. The two summands
have images $(P_1)_o$ and $(P_2)_o$: trivial summands contribute no
sections after twisting, while positive summands allow arbitrary node
values. This proves the second isomorphism.
Explicitly, lift a pair of values at the marked points to sections
$s_i\in H^0(C_i,T_i')$, possible by freeness. Its cokernel class is
represented by $s_1(o)-s_2(o)$ modulo $(P_1)_o+(P_2)_o$; changing the
lifts does not change this class.
\end{proof}

For $f_0$, the condition \eqref{cf:pair-conditions} identifies this
cokernel with $T_yY/D_y\simeq L_y$ via $\theta_y$.
Since $d\sigma$ is surjective, $d\operatorname{Ev}$ induces
\begin{equation}\label{cf:missing-evaluation-map}
\overline{d\operatorname{Ev}}:T_0\mathbb A^1
\longrightarrow
\operatorname{coker}\bigl(d\operatorname{Ev}|_{\ker d\sigma}\bigr)
\simeq L_y.
\end{equation}
It sends $d\sigma(\xi)$ to the class of $d\operatorname{Ev}(\xi)$.
Both sides are one-dimensional, so $d\operatorname{Ev}$ is surjective
exactly when this map is nonzero. The next calculation identifies it
with the Levi pairing of the two branch tangents.

\begin{lemma}\label{cf:fixed-end-obstruction}
With the node-difference convention above,
\[
\overline{d\operatorname{Ev}}(\partial_s)=\omega_y(a,b).
\]
In particular, \eqref{cf:missing-evaluation-map} is an isomorphism.
\end{lemma}

\begin{proof}
Choose $\zeta\in T_{[f_0]}\mathscr H$ with $d\sigma(\zeta)=\partial_s$.
Over $B_\varepsilon=\operatorname{Spec}\CC[\varepsilon]/(\varepsilon^2)$,
it represents a deformation $F_\varepsilon:\mathscr C_\varepsilon\to Y$
of $f_0$, where $\mathscr C_\varepsilon$ is obtained by setting
$s=\varepsilon$, so its node equation is $uv=\varepsilon$.
The two projections trivialize this source away from the node on the
respective branches, including their marked points at infinity; in the
affine charts they are given by $v=\varepsilon/u$ and $u=\varepsilon/v$.
Using these trivializations, define the first-order variations $A,B$
by the expansions in local target coordinates
\[
\begin{aligned}
F_\varepsilon(u,\varepsilon/u)&=f_1(u)+\varepsilon A(u),\\
F_\varepsilon(\varepsilon/v,v)&=f_2(v)+\varepsilon B(v).
\end{aligned}
\]
Intrinsically, $A,B$ are rational sections of $T_1,T_2$,
respectively, regular away from the node.
Their values at infinity give the first-order variations of the two
marked images, so $d\operatorname{Ev}(\zeta)=(A(\infty),B(\infty))$.
By freeness, choose $s_i\in H^0(C_i,T_i)$ with
\[
s_1(\infty)=A(\infty),\qquad s_2(\infty)=B(\infty).
\]
The description in Lemma~\ref{cf:evaluation-cokernel} gives
\[
\overline{d\operatorname{Ev}}(\partial_s)
=\theta_y\bigl(s_1(0)-s_2(0)\bigr).
\]
We compute this difference from the two branch expansions.

Choose common target coordinates centred at $y$. In the completed
local rings, write the coordinate expansions of the branch maps as
$g_1(u),g_2(v)$, with zero constant terms. The deformation has a
formal expression
\[
g_1(u)+g_2(v)+\varepsilon R(u,v)
\quad\text{in }\CC[[u,v,\varepsilon]]/(uv-\varepsilon,\varepsilon^2).
\]
Substitution on the punctured branches, using $\varepsilon^2=0$, gives
\begin{equation}\label{cf:common-constant}
\begin{aligned}
A(u)&=\frac b u+R(u,0)=\frac b u+c+O(u),\\
B(v)&=\frac a v+R(0,v)=\frac a v+c+O(v),
\end{aligned}
\end{equation}
where $c=R(0,0)$ is the same in both expansions.

Set $W_1=u(A-s_1)$ and $W_2=v(B-s_2)$. These sections extend across
the node by \eqref{cf:common-constant}; the zeros of $A-s_1$ and
$B-s_2$ at infinity cancel the simple poles of $u,v$. Thus
$W_i\in H^0(C_i,T_i)$, with node jets
\[
\begin{aligned}
W_1(0)&=b,& W_1'(0)&=c-s_1(0),\\
W_2(0)&=a,& W_2'(0)&=c-s_2(0).
\end{aligned}
\]
Each $W_i$ defines a first-order deformation
$G_i:C_i\times B_\varepsilon\to Y$ of $f_i$.
Since $H^1(C_i,\OO_{C_i})=0$, the line bundle $f_i^*L\simeq\OO(1)$
has only the trivial first-order deformation, so
$G_i^*L\simeq\OO_{C_i\times B_\varepsilon}(1)$.
The relative contact differential of $G_i$ is therefore zero, since
$H^0(C_i\times B_\varepsilon,\OO(-1))=0$.
Write $\theta=\alpha\otimes e$ in a local frame near $y$.
The coefficients of $\varepsilon$ in these horizontality identities
at the node are
\[
\begin{aligned}
(\partial_b\alpha)_y(a)+\alpha_y\bigl(c-s_1(0)\bigr)&=0,\\
(\partial_a\alpha)_y(b)+\alpha_y\bigl(c-s_2(0)\bigr)&=0.
\end{aligned}
\]
Here $\partial_w\alpha$ differentiates the coefficients of $\alpha$
in the direction $w$. Subtracting gives
\[
\theta_y\bigl(s_1(0)-s_2(0)\bigr)
=-d\alpha_y(a,b)e_y=\omega_y(a,b).
\]
This proves the formula, and \eqref{cf:pair-conditions} makes it nonzero.
\end{proof}

\begin{remark*}
For two free maps of $L$-degree one meeting at $y$ in a smooth
projective variety carrying an $L$-valued one-form $\theta$, the same
calculation shows that a first-order smoothing fixing the values at both
marked points forces $d\alpha_y(a,b)=0$. Here $\theta=\alpha\otimes e$
locally and $a,b$ are the branch tangent vectors. This obstruction
does not require a contact structure.
\end{remark*}

\begin{proposition}[Smoothing to a balanced map]\label{cf:first-conic}
There exists a morphism $g:\PP^1\to Y$ of contact degree two
with the following properties.
\begin{enumerate}
\renewcommand{\labelenumi}{\textup{(\arabic{enumi})}}
\item\label{cf:first-conic-bundle}
The map $g$ is balanced.
\item\label{cf:first-conic-evaluation}
The evaluation morphism at the fixed points
$0,\infty\in\PP^1$,
\[
\epsilon: \underline{\operatorname{Hom}}(\PP^1,Y)\longrightarrow Y\times Y,\qquad
[h]\longmapsto\bigl(h(0),h(\infty)\bigr),
\]
is dominant on the unique irreducible component containing $[g]$.
\end{enumerate}
\end{proposition}

\begin{proof}\leavevmode\par
\noindent\textup{(1)}
Lemma~\ref{cf:fixed-end-obstruction} makes
\eqref{cf:missing-evaluation-map} an isomorphism, so
$d\operatorname{Ev}$ is surjective at $[f_0]$. Since $\mathscr H$ and
$Y\times Y$ are smooth there, $\operatorname{Ev}$ is smooth at $[f_0]$.

Choose a smooth connected quasi-projective open neighbourhood $\mathscr U$ of $[f_0]$ in
$\mathscr H$ on which both $\sigma$ and $\operatorname{Ev}$ are smooth.
Since $\sigma$ is open, $\mathscr U$ meets
$\sigma^{-1}(\mathbb A^1_s\setminus\{0\})$, where the source is a smooth
$\PP^1$. The universal source family is projective and flat over
$\mathscr U$, so by
Lemma~\ref{cf:deformation-facts}\textup{(\ref{cf:degree-constancy})}
all its maps have contact degree
$\deg f_1^*L+\deg f_2^*L=2$.
Over $\mathbb A^1_s\setminus\{0\}$, the source family $\mathscr C$
admits the trivialization
\[
\begin{aligned}
\PP^1\times(\mathbb A^1_s\setminus\{0\})
&\xrightarrow{\sim}\mathscr C|_{s\ne0},\\
([U_0:U_1],s)&\longmapsto
([U_0:U_1],[U_1:sU_0],s).
\end{aligned}
\]
Composing maps $f:\mathscr C_s\to Y$ with this trivialization identifies
the part of $\mathscr U$ over $s\ne0$ with an open subscheme of
$\underline{\operatorname{Hom}}_2(\PP^1,Y)\times(\mathbb A^1_s\setminus\{0\})$.
Choose a point $([g],s)$ there. Since $p,q$ correspond to $\infty,0$,
$\operatorname{Ev}([g],s)=(g(\infty),g(0))$ is independent of $s$.
Thus surjectivity of $d\operatorname{Ev}$ gives a surjection
\[
H^0(\PP^1,g^*TY)\longrightarrow
T_{g(0)}Y\oplus T_{g(\infty)}Y.
\]
A summand of nonpositive degree has a two-point evaluation map of
rank at most one, so $g$ is very free. Since its contact degree is two,
it is balanced by the observation following
Definition~\ref{cf:balanced-definition}.

\noindent\textup{(2)} Since $H^1(g^*TY)=0$,
Lemma~\ref{cf:deformation-facts}\textup{(\ref{cf:hom-smoothness})}
shows that $\underline{\operatorname{Hom}}(\PP^1,Y)$ is smooth at $[g]$.
Thus $[g]$ lies on a unique irreducible component. By the preceding
surjectivity, evaluation at $0,\infty$ is smooth at $[g]$; its image
on that component therefore contains a nonempty open subset of
$Y\times Y$, proving dominance.

\end{proof}

\begin{remark}\label{cf:other-characteristics}
Proposition~\ref{cf:first-conic} also holds for a smooth connected
projective contact variety over an algebraically closed field of arbitrary
characteristic, provided it admits two free contact lines through a
common point satisfying \eqref{cf:pair-conditions}. The proof uses
only deformation-theoretic $H^1$-vanishing and the first-order
calculation of Lemma~\ref{cf:fixed-end-obstruction}, both valid in
arbitrary characteristic. This observation concerns only
the smoothing result; the subsequent propagation and homogeneity
arguments are carried out in characteristic zero.
\end{remark}

\subsection{Degree-two curves through every point}\label{cf:sec-propagation}

The two-line boundary bound below verifies a general criterion for
propagating very free curves with tangent lines spanning each tangent
space. In the contact setting, \eqref{cf:degree-relation} and
Lemma~\ref{cf:degree-two-transversality} then give balanced transverse maps.

\begin{lemma}\label{cf:two-line-boundary}
For every $y\in Y$, the set $R_y$ of points reachable from $y$ by a
chain of at most two contact lines is closed and has dimension at most
$2n$. We include $y$ itself in $R_y$.
\end{lemma}

\begin{proof}
Since $L$ is ample and contact lines have $L$-degree one, they cannot
split. Their Chow locus, with its reduced structure, is a projective scheme
\cite[Proposition~3.7(3)]{BKK}, with finitely many irreducible components.
The fibre square of the incidence support of its universal cycle parametrizes triples
$(C,z,w)$ with $z,w\in C$. Evaluation $(C,z,w)\mapsto(z,w)$ is proper,
so its image
\[
R=\{(z,w)\in Y\times Y:\text{$z,w$ lie on a contact line}\}
\]
is closed.

For each $y\in Y$, let $\Lambda_y=\{z\in Y:(y,z)\in R\}$ be the
union of contact lines through $y$. It is closed, and
\cite[Proposition~5.8]{BKK}, applied to the finitely many components,
gives $\dim\Lambda_y\le n$ for every $y$, including special points.

The projection of $R\cap(\Lambda_y\times Y)$ to $\Lambda_y$ has
fibre $\Lambda_z$ over $z$, so its total dimension is at most $2n$.
Its image under the second projection is closed by properness and
consists of points reachable from $y$ by two contact lines.
Adding $\Lambda_y\cup\{y\}$ gives $R_y$ and proves the assertion.
\end{proof}

\begin{remark*}
The same argument applies to any family of curves on a projective
variety $X$ with proper two-point incidence evaluation: if the union
$\Lambda_x$ of its members through $x$ satisfies $\dim\Lambda_x\le a$
for every $x\in X$, with $a\ge0$, then the locus reachable from $x$
by at most two members is closed of dimension at most $2a$, and is
a proper subset of $X$ whenever $2a<\dim X$.
\end{remark*}

\begin{proposition}[Very free curves through every point]\label{cf:very-free-propagation}
Let $X$ be a smooth projective variety and let
$0\ne\beta\in H_2(X,\ZZ)$ be a curve class. Let
$\overline M\subset\overline M_{0,2}(X,\beta)$ be a closed irreducible
subvariety of the coarse moduli space of stable maps, and write
\[
e=(e_1,e_2):\overline M\longrightarrow X\times X
\]
for two-point evaluation. Let $\partial\overline M\subset\overline M$
be the locus of stable maps whose domain curves are singular, with
its reduced induced structure. It is closed, since the smooth-domain locus is open in
$\overline M_{0,2}(X,\beta)$. Suppose that $e$ is dominant
and that
\[
B_x:=e_2\bigl(\partial\overline M\cap e_1^{-1}(x)\bigr)\ne X
\qquad\text{for every }x\in X.
\]
Then every $x\in X$ lies on a very free map $f:\PP^1\to X$ of class
$\beta$, with $f(0)=x$. For a fixed local coordinate $u$ at $0$, the
vectors $df_0(\partial_u)$ arising from such maps form a nonempty open
subset of $T_xX$. In particular, $\dim X$ such maps can be chosen
whose tangent lines at $x$ span $T_xX$.
\end{proposition}

\begin{proof}
The variety $\overline M$ is projective \cite{BM,FP}, so $e$ is proper
and hence surjective. Fix $x\in X$. The set $B_x$ is closed, and for
every $y\in X\setminus B_x$ a stable map over $(x,y)$ has a smooth domain
curve. Since its genus is zero, the domain is $\PP^1$, and we may place
its markings at $0$ and $1$. Consequently evaluation at $1$ on
\[
\underline{\operatorname{Hom}}_\beta(\PP^1,X;0\mapsto x)
\]
covers $X\setminus B_x$. This is a closed subscheme of the
quasi-projective finite-type scheme
$\underline{\operatorname{Hom}}_\beta(\PP^1,X)$.
Thus some reduced irreducible component $H$, a quasi-projective
variety, has evaluation at $1$ dominant onto $X$.

By characteristic-zero generic smoothness, at a general smooth point
$[h]\in H$ its evaluation differential is surjective. It factors as
\[
T_{[h]}H\hookrightarrow H^0(\PP^1,h^*TX(-[0]))
\longrightarrow T_{h(1)}X.
\]
Thus the section-evaluation map is surjective. If a splitting summand
of $h^*TX$ had nonpositive degree, it would contribute no sections
after twisting by $-[0]$, contradicting surjectivity. Every summand
therefore has positive degree, so $h$ is very free.

Very freeness is equivalent to $H^1(f^*TX(-2[0]))=0$, hence defines
an open locus by semicontinuity. By
Lemma~\ref{cf:deformation-facts}\textup{(\ref{cf:jet-smoothness})},
the derivative map $[f]\mapsto df_0(\partial_u)$ on
$\underline{\operatorname{Hom}}_\beta(\PP^1,X;0\mapsto x)$ is smooth
at every very free map. Its image on the nonempty very free locus
is therefore open in $T_xX$, and contains a basis of $T_xX$.
\end{proof}

\begin{corollary}\label{cf:every-point}
Every point $y\in Y$ lies on a balanced transverse map.
Moreover, $2n+1$ such maps through $y$ can be chosen whose tangent
lines at $y$ span $T_yY$.
\end{corollary}

\begin{proof}
Choose $f_0:\PP^1\to Y$ as in Proposition~\ref{cf:first-conic}, and let
$\beta=(f_0)_*[\PP^1]\in H_2(Y,\ZZ)$ be its curve class.
Map the Hom component through $[f_0]$ to the coarse moduli space
$\overline M_{0,2}(Y,\beta)$ by $[h]\mapsto[\PP^1,0,\infty,h]$,
and let $\overline M$ be the reduced closure of its image.
This is a projective variety \cite{BM,FP}. The map preserves
two-point evaluation, so
Proposition~\ref{cf:first-conic}\textup{(\ref{cf:first-conic-evaluation})}
implies that $e=(e_1,e_2):\overline M\to Y\times Y$ is dominant.

Fix $y\in Y$. We claim that the boundary image $B_y$ of
Proposition~\ref{cf:very-free-propagation} is contained in $R_y$.
Consider a stable map in this boundary. Its domain is nodal of
arithmetic genus zero, so its dual graph is a tree; since the domain
is singular, it is reducible. The total contact degree is two, so
ampleness of $L$ allows at most two nonconstant components. If there
are two, their images form a chain of contact lines containing both
marked images. Otherwise there is one nonconstant component, and
reducibility of the domain forces a contracted component. A maximal
contracted tree attaches once to the nonconstant component.
If it has $r\ge1$ vertices and $k$ markings, stability gives
\[
2(r-1)+1+k\ge3r,\qquad\text{hence~} k\ge r+1\ge2.
\]
Both markings then have the same image $y$. Thus $B_y\subset R_y$,
which is a proper closed subset of $Y$ by
Lemma~\ref{cf:two-line-boundary}.

The open subset of $T_yY$ supplied by
Proposition~\ref{cf:very-free-propagation} contains a basis outside
$D_y$. Choose corresponding very free maps $f_1,\ldots,f_{2n+1}$
of class $\beta$ with $f_i(0)=y$. Their contact degree is two, so
they are balanced. Each is transverse at $0$, hence everywhere
transverse by Lemma~\ref{cf:degree-two-transversality}.
\end{proof}

Corollary~\ref{cf:every-point} supplies the balanced transverse maps
used in the next section. The two points in the incidence construction
always belong to the source; injectivity of the maps is not assumed.

\section{Tangent sections and contact homogeneity}\label{cf:sec-tangent-sections}

\subsection{General tools for rational curves}\label{cf:sec-general-tools}

The results in this subsection do not require a contact structure.
We begin with an extension lemma for line-bundle-valued relative
differentials on families of rational curves.

\noindent\textbf{The two-point tensor and nodal extension.}
Let $p:\mathcal C\to S$ be a proper flat morphism of schemes with nodal
curve fibres, where $S$ is reduced and of finite type over $\CC$.
Suppose that every geometric
fibre is either $\PP^1$ or two projective lines meeting at one node, and
that the smooth-fibre locus $S^\circ$ is dense in every component of $S$.
Let $\mathcal L$ be a line bundle of degree two on smooth fibres and
degree one on each component of a reducible fibre. Write
$\omega=\omega_{\mathcal C/S}$ and
$\rho:\Omega^1_{\mathcal C/S}\to\omega$ for the canonical comparison map.
The sheaf $\omega$ is invertible, restricting to $\OO_{\PP^1}(-2)$ on a
smooth fibre and to $\omega_{C_i}(q_i)\simeq\OO_{\PP^1}(-1)$ on each
component $C_i$ of a nodal fibre, where $q_i$ is the point over the node.
The map $\rho$ is an isomorphism on the smooth locus of $p$.
Set $\mathcal C^\circ=p^{-1}(S^\circ)$, and let
$\mathrm{pr}_a,\mathrm{pr}_b:\mathcal C\times_S\mathcal C\to\mathcal C$
be the two projections. We use the following two facts:
\begin{enumerate}
\renewcommand{\labelenumi}{\textup{(\arabic{enumi})}}
\item The smooth family $\mathcal C^\circ\to S^\circ$ is
\'etale-locally trivial with fibre $\PP^1$.
\item For an affine coordinate $u$ on $\PP^1$, the tensor
\[
q(a,b)=(u(a)-u(b))^2\partial_{u,a}\otimes\partial_{u,b}
\]
is regular on $\PP^1\times\PP^1$ and invariant under simultaneous
fractional linear changes of coordinate in the two factors.
\end{enumerate}
Thus the local tensors agree on overlaps and descend to a regular
section of $\mathrm{pr}_a^*\omega^{-1}\otimes\mathrm{pr}_b^*\omega^{-1}$
on $\mathcal C^\circ\times_{S^\circ}\mathcal C^\circ$. On each smooth
fibre product, its zero divisor is twice the diagonal.
This open set is dense in every component of the
reduced two-point family $\mathcal C\times_S\mathcal C$, so $q$
defines a canonical rational section on the entire two-point family.

\begin{lemma}\label{cf:cancellation}
In this setting, for every $\vartheta\in
H^0(\mathcal C,\Omega^1_{\mathcal C/S}\otimes\mathcal L)$, the tensor product
\[
\mathrm{pr}_a^*((\rho\otimes1)\vartheta)\otimes q
\]
extends uniquely to a regular section of
$\mathrm{pr}_a^*\mathcal L\otimes\mathrm{pr}_b^*\omega^{-1}$ on
$\mathcal C\times_S\mathcal C$.
Here we use the canonical identification
$\mathrm{pr}_a^*\omega\otimes\mathrm{pr}_a^*\omega^{-1}\simeq\OO$.
\end{lemma}

\begin{proof}
Since $\mathcal L\otimes\omega$ has multidegree zero on each rational
tree fibre, cohomology and base change give a line bundle
$B=p_*(\mathcal L\otimes\omega)$ and an evaluation isomorphism
\[
p^*B\simeq\mathcal L\otimes\omega.
\]
Also $p_*\OO_{\mathcal C}=\OO_S$, so
$(\rho\otimes1)\vartheta=p^*s$ for a unique $s\in H^0(S,B)$.

On each fibre, $\omega^{-1}$ is very ample with $h^0=3$ and $h^1=0$.
Since $p$ is proper and flat, cohomology and base change give a
rank-three vector bundle $E=p_*\omega^{-1}$, compatible with base change.
Evaluation defines the relative anticanonical embedding
\[
\Phi:\mathcal C\hookrightarrow\PP_S(E^\vee).
\]
In particular, $p$ is projective over $S$.
See also \cite[Proposition~2.7]{Tanaka} for details of this construction.
Similarly, $p_*\omega^{-2}$ is locally free of rank five.
The restriction map on quadratic forms
$\operatorname{Sym}^2E\to p_*\omega^{-2}$ is surjective on every fibre,
hence surjective with a line-bundle kernel. A local generator defines
a flat quadratic hypersurface containing $\Phi(\mathcal C)$ and
agreeing with it scheme-theoretically on every fibre.
Since $\mathcal C$ is flat over $S$, the ideal sequence of
$\Phi(\mathcal C)$ in this hypersurface remains exact after taking
fibres. Fibrewise equality and Nakayama's lemma therefore make the
ideal vanish.

Choose an affine neighbourhood $\operatorname{Spec}A\subset S$ of a
point with nodal fibre, trivializing $E$ and the line of quadratic
relations. Near that fibre, the nondegenerate binary part of the
quadratic equation splits after an affine \'etale base change
$\operatorname{Spec}R\to\operatorname{Spec}A$. Taking its factors as
coordinates and eliminating the $XZ$ and $YZ$ terms gives
\[
\mathcal C_R:=\mathcal C\times_S\operatorname{Spec}R
\simeq\{XY=tZ^2\}\subset\PP^2_R,\qquad t\in R.
\]
All bundles and sections in the following local calculation are
understood after this base change. Here $t$ is the actual smoothing
element, which need not be a parameter on the base.
Density of $S^\circ$ implies that $t$ does not vanish identically
on any component of $\operatorname{Spec}R$. Since $R$ is reduced,
$t$ is a nonzerodivisor.

On $Z\ne0$, write $xy=t$ and choose the dualizing generator $\eta$ so that
\[
\rho(dx)=x\eta,\qquad \rho(dy)=-y\eta.
\]
Trivialize $B$ and write $s=re$ with $r\in R$. Since $p^*s$ lies in
the image of $\rho\otimes1$, its coefficient belongs to $(x,y)$ in
$R[x,y]/(xy-t)$. Reduction modulo $(x,y)$ gives $r=0$ in $R/(t)$.
Thus $s=ts_1$ with $s_1$ regular, even when $R/(t)$ is nonreduced.

Adjunction identifies $\OO_{\mathcal C_R}(1)$ with $\omega^{-1}$.
The sections corresponding to $X,Y,Z$ can be
normalized on $Z\ne0$ as
\[
e_x=x\eta^\vee,\qquad e_y=y\eta^\vee,\qquad e_0=\eta^\vee.
\]
On the chart $Z\ne0$ over $t\ne0$, $x$ is a fibre coordinate.
Since $\eta=dx/x$, its dual is $\eta^\vee=x\partial_x$.
Substituting into the formula for $q$ gives
\[
tq=e_{x,a}\otimes e_{y,b}+e_{y,a}\otimes e_{x,b}
-2t\,e_{0,a}\otimes e_{0,b}=:Q.
\]
Since $e_x,e_y,e_0$ are global regular sections of $\omega^{-1}$,
$Q$ is regular on the entire two-point family. Writing
$(\rho\otimes1)\vartheta=t\widetilde\vartheta_1$ therefore extends
the required section as
\[
\mathrm{pr}_a^*\widetilde\vartheta_1\otimes Q.
\]
The two-point family is reduced, since it is flat over a reduced base
with geometrically reduced fibres. Flatness also makes every
irreducible component dominate a component of $S$. Thus the inverse
image of $S^\circ$ is dense in every component, and the extensions
agree on overlaps and descend uniquely.
\end{proof}

\noindent\textbf{First jets and preservation of tangent directions.}
We next relate global sections on a projective source to the tangent
directions of rational curves on the target.

Let $X$ be a smooth projective variety, and let
$\mathcal K\subset\RatCurves^n(X)$ be an irreducible dominating
component of the normalized space of rational curves; thus
$\mathcal K$ is a normal quasi-projective variety.
Let $\mathcal V\subset\PP(TX)$ be the reduced closure of
\[
\bigl(\nu(t),[d\nu_t(T_t\PP^1)]\bigr),
\]
where $\nu$ ranges over birational parametrizations of members of
$\mathcal K$ and $d\nu_t\ne0$. It is an integral projective variety:
it is the reduced image closure of the rational tangent map from the
irreducible normalized universal family. Write $p:\mathcal V\to X$
for the projection.

For a vector field $v$ on a smooth variety $E$, let $\widehat v$ be
its natural lift to $\PP(TE)$, obtained by differentiating its
infinitesimal automorphism. Equivalently, the Lie derivative on
$\Omega_E^1$ induces a degree-preserving derivation of
$\operatorname{Sym}\Omega_E^1$, and hence a vector field on
$\PP(TE)$. In local coordinates $x_i$, if
$v=\sum v_i\partial_{x_i}$, the action on one-forms is
\[
\mathcal L_v(dx_i)=d(v_i)
=\sum_j\frac{\partial v_i}{\partial x_j}\,dx_j.
\]
A vector field is tangent to a reduced
closed subscheme if it preserves its ideal sheaf.

For a vector bundle $F$ on a smooth curve $C$, the first jet of a
section $s$ at $p$ is its restriction $j_p^1(s)=s|_{2p}$, where
$2p=\operatorname{Spec}(\OO_{C,p}/\mathfrak m_p^2)$.
In a local coordinate and a local frame, this records its value and
first derivative at $p$.

\begin{lemma}\label{cf:firstjet}
Let $E$ be a vector bundle on $\PP^1$, and let $g:B\to\PP^1$ be a
finite surjective morphism from a smooth projective curve, unramified
at $b$. Every $s\in H^0(B,g^*E)$ has the same first jet at $b$ as
$g^*s_0$ for some $s_0\in H^0(\PP^1,E)$.
\end{lemma}

\begin{proof}
Split $E$ into line bundles $\OO(a)$. Negative-degree summands have
no sections after pullback. Degree-zero summands have constant
sections, since $B$ is projective. Put $p=g(b)$. For $a\ge1$, the
exact sequence
\[
0\longrightarrow\OO(a)(-2p)\longrightarrow\OO(a)
\longrightarrow\OO(a)|_{2p}\longrightarrow0
\]
and $H^1(\PP^1,\OO(a-2))=0$ show that global sections realize
arbitrary first jets at $p$.
Unramifiedness identifies these with first jets at $b$.
\end{proof}

\begin{proposition}\label{cf:jets}
Let $f:W\to X$ be a surjective generically finite morphism from a
projective variety, and let $S\in H^0(W,f^*TX)$. On any open subset
$E\subset W$ where $f$ is \'etale, put $v=(df)^{-1}(S|_E)$.
Under the identification $\PP(TE)=\PP(TX)\times_XE$, the vector field
$\widehat v$ is tangent to $\mathcal V_E:=\mathcal V\times_XE$.
Here $E$ is smooth and quasi-projective, and $\mathcal V_E$ is a
reduced quasi-projective scheme, possibly reducible.
\end{proposition}

\begin{proof}
\emph{Reduction to general tangent directions.}
Since $\mathcal V_E$ is reduced, it suffices to prove tangency on a
dense subset. Tangent directions of general members of $\mathcal K$
at general points are dense in $\mathcal V$; their lifts are dense in
$\mathcal V_E$, since $\mathcal V_E\to\mathcal V$ is \'etale.
We will realize $\widehat v$ at such a direction as the induced
velocity of a curve deformation within $\mathcal K$.

Choose a general birational parametrization $\nu:\PP^1\to X$ of a
member of $\mathcal K$, a general point $t_0$ with $d\nu_{t_0}\ne0$,
and $e\in E$ over $\nu(t_0)$. Since $\mathcal K$ is dominating,
$\nu$ is free \cite[II.3.10]{Kollar}.

\smallskip
\noindent \emph{Matching the first jet.}
To apply Lemma~\ref{cf:firstjet}, let $B$ be the normalization of the
reduced irreducible component of $W\times_X\PP^1$ through $(e,t_0)$,
with projections $g:B\to\PP^1$ and $j:B\to W$.
Since $f$ is \'etale at $e$, this component is a curve dominating
$\PP^1$. Thus $B$ is a smooth projective curve, $g$ is finite and
surjective, and it is unramified at the point $b$ over $(e,t_0)$.
The lemma gives $s_0\in H^0(\PP^1,\nu^*TX)$ such that
\[
j_b^1(g^*s_0)=j_b^1(j^*S).
\]
Over $j^{-1}(E)$, the identity $df(v)=S$ identifies $j^*S$ with $v$
along the lifted curve.

\smallskip
\noindent \emph{Deforming within the component.}
Freeness gives $H^1(\PP^1,\nu^*TX)=0$, so
Lemma~\ref{cf:deformation-facts}\textup{(\ref{cf:hom-smoothness})}
realizes $s_0$ as the velocity of a deformation $\nu_\lambda$ over a
smooth pointed algebraic curve. Because $\mathcal K$ is an
\emph{irreducible component} and $\nu$ is general, the deformed image
curves remain in $\mathcal K$ after shrinking the parameter curve. Consequently
\[
\bigl(\nu_\lambda(t_0),
[d\nu_{\lambda,t_0}(T_{t_0}\PP^1)]\bigr)\in\mathcal V.
\]
The velocity of this curve of tangent directions depends only on the
first jet of $s_0$. Formal \'etaleness lifts its first-order deformation
uniquely to $\mathcal V_E$ through the chosen point over $e$.
The jet equality identifies this lifted velocity with $\widehat v$,
proving tangency at the chosen general direction.

Finally, for any local function $h$ vanishing on $\mathcal V_E$,
the constructed deformations give $\widehat v(h)=0$ on the dense set
above. Reducedness gives $\widehat v(h)|_{\mathcal V_E}=0$ everywhere,
so $\widehat v$ preserves the ideal of $\mathcal V_E$.
\end{proof}

\noindent\textbf{Smoothness from lifted vector fields.}
We next use infinitesimal preservation to remove the nonsmooth locus
of the projection. This does not require integrating a vector field.

\begin{lemma}\label{cf:smoothness}
Let $p:Z\to X$ be a proper dominant morphism of varieties, with
$X$ smooth. Suppose $p$ is smooth over a nonempty open subset of
$X$. At every $x\in X$, assume there are pointed \'etale neighbourhood schemes
$(E_i,e_i)\to(X,x)$ and vector fields $v_i$ on $E_i$ such that:
\begin{enumerate}
\item each $v_i$ lifts to a vector field on all of $Z\times_XE_i$,
compatible with the projection to $E_i$;
\item the images of $v_i(e_i)$ span $T_xX$.
\end{enumerate}
Then $p$ is smooth.
\end{lemma}

\begin{proof}
\emph{The nonsmooth locus and its image.}
Let $B\subset Z$ be the reduced closed subscheme supported on the
nonsmooth locus of $p$. Properness of $p$ makes $A=p(B)$ closed;
give it the reduced induced scheme structure.
Smoothness over a nonempty open subset gives $A\ne X$.
It suffices to prove that $A$ is empty.

\smallskip
\noindent \emph{Preservation by the vector fields.}
We show that the lifted fields preserve $B$, so the fields on the
base preserve $A$. Put $r=\dim Z-\dim X$. The locus $B$ is cut out
set-theoretically by the Fitting ideal
\[
J=\Fitt_r(\Omega^1_{Z/X}).
\]
Indeed, at a closed point $z$ its complement is the condition
$\dim\ker(dp_z)\le r$. Since $\dim T_zZ\ge\dim Z$, this forces
$Z$ to be smooth at $z$ and $dp_z$ to be surjective. The converse
holds at every smooth point of $p$. This also applies after \'etale
base change.

Over $\CC[\epsilon]/(\epsilon^2)$, the compatible derivations
$\widetilde v,v$ define automorphisms $1+\epsilon\widetilde v$ and
$1+\epsilon v$ commuting with the projection, hence an isomorphism
of relative cotangent sheaves. Fitting ideals commute with base change
and are invariant under these isomorphisms, so $1+\epsilon\widetilde v$
preserves the extended ideal $J_\epsilon:=J\oplus\epsilon J$.
For any local section $h\in J$, this gives
$h+\epsilon\widetilde v(h)\in J_\epsilon$; comparing
$\epsilon$-coefficients yields $\widetilde v(h)\in J$.
Thus $\widetilde v(J)\subset J$. In characteristic zero, a derivation
preserving an ideal also preserves its radical. Hence each lifted
field preserves the reduced nonsmooth locus $B_{E_i}$.

Consequently each $v_i$ preserves $A_{E_i}$: a local function vanishing
there pulls back to zero on $B_{E_i}$, and by compatibility the same
holds for its derivative. Since $B_{E_i}\to A_{E_i}$ is surjective and $A_{E_i}$
is reduced, that derivative vanishes on $A_{E_i}$.

\smallskip
\noindent \emph{The spanning condition.}
If $A$ were nonempty, choose a smooth point $x$ of $A$. Preservation
of $A_{E_i}$ implies that every $v_i(e_i)$ maps into
$T_xA\subsetneq T_xX$. This contradicts the assumption that these
values span $T_xX$. Thus $A$, and hence $B$, is empty, so $p$ is smooth.
\end{proof}

\subsection{Incidence varieties and tangent sections}\label{cf:sec-incidence}

The incidence construction of \cite{HLW} associates tangent sections to
balanced transverse curves. Using Lemma~\ref{cf:cancellation} to extend
these sections across the nodal boundary, we prove the version for
balanced transverse maps of Definition~\ref{cf:balanced-definition},
keeping track of points on the source.

Let $Y$ be a smooth projective contact variety of dimension $2n+1$, with
$n\geq2$, contact form $\theta:TY\to L$, and ample $L$.
Fix a balanced transverse map $f_0:\PP^1\to Y$, a source
point $a_0$, and $z=f_0(a_0)$. Set $\beta=(f_0)_*[\PP^1]$.
Let $\mathcal M$ be the reduced irreducible component of
$\overline{\mathcal M}_{0,0}(Y,\beta)$ through $[f_0]$.
It is a proper integral Deligne--Mumford stack, with projective
coarse moduli variety $\overline M$ \cite[Theorem~1]{FP}. Write
\[
p:U\longrightarrow\mathcal M,\qquad F:U\longrightarrow Y
\]
for its universal domain and evaluation map. The morphism $p$ is
representable, proper, flat and nodal, and $U$ is a proper
Deligne--Mumford stack. The component is unique at $[f_0]$, since
$H^1(\PP^1,f_0^*TY)=0$.
Let $\mathcal M^\circ$ be the open locus of balanced transverse maps,
and put $U^\circ=U|_{\mathcal M^\circ}$.
Such maps are birational onto their images: factoring through the
normalization and applying Riemann--Hurwitz rules out a nontrivial
unramified cover of $\PP^1$. Their stable-map automorphisms are therefore
trivial, so $\mathcal M^\circ$ is identified with an open subscheme of
$\overline M$. The vanishing $H^1(f^*TY)=0$ makes it a smooth
quasi-projective variety. The family $U^\circ\to\mathcal M^\circ$ is
smooth and projective: the relative anticanonical bundle embeds it
into a $\PP^2$-bundle. Thus $U^\circ$ is also a smooth quasi-projective
variety.

The fibre of $F|_{U^\circ}$ over $z$ is smooth at $(f,a)$, since
$H^1(f^*TY(-[a]))=0$
(Lemma~\ref{cf:deformation-facts}\textup{(\ref{cf:evaluation-smoothness})}).
Let $B_z$ be the unique reduced irreducible component of
$(F|_{U^\circ})^{-1}(z)$ through $(f_0,a_0)$.
By \eqref{cf:pointed-moduli-tangent}, for $(f,a)\in B_z$,
\[
T_{(f,a)}B_z\simeq
\frac{H^0(f^*TY(-[a]))}{df\,H^0(T\PP^1(-[a]))}.
\]
The balanced splitting gives
$f^*TY(-[a])\simeq\OO(1)\oplus\OO^{\oplus2n}$, hence
$\dim B_z=(2+2n)-2=2n$. Let
\[
\mathcal U_z=B_z\times_{\mathcal M}U\longrightarrow B_z
\]
be its universal source curve. Both are smooth quasi-projective
varieties, and $\mathcal U_z\to B_z$ is projective; points of
$\mathcal U_z$ are triples $(f,a,b)$ with $f(a)=z$.

\begin{theorem}[Incidence variety and tangent section]\label{cf:incidence}
There exist a normal projective variety $W_z$ containing
$\mathcal U_z$ as a dense open subset, a surjective generically finite
morphism $e_z:W_z\to Y$, and a section
\[
S_z\in H^0(W_z,e_z^*TY)
\]
with the following properties.
\begin{enumerate}
\renewcommand{\labelenumi}{\textup{(\arabic{enumi})}}
\item\label{cf:incidence-evaluation}
\textnormal{(Evaluation.)} On $\mathcal U_z$, $e_z(f,a,b)=f(b)$, and
$e_z$ is \'etale at $(f,a,b)$ whenever $b\ne a$.
\item\label{cf:incidence-tangency}
\textnormal{(Motion along each source curve.)} For each $(f,a)\in B_z$,
identify the fibre of $\mathcal U_z\to B_z$ over $(f,a)$ with the source
$\PP^1$ of $f$. The restriction of $S_z$ to this fibre is $df(V_{f,a})$,
where $V_{f,a}$ is a vector field on $\PP^1$ with a double zero at $a$
and no other zeros. In particular,
\[
0\ne S_z(f,a,b)\in df_b(T_b\PP^1)\qquad(b\ne a).
\]
\end{enumerate}
Here $b\ne a$ refers to distinct source points; injectivity of $f$ is not assumed.
\end{theorem}

\begin{proof}
\textit{Compactification.}
Transversality gives $L\cdot\beta=2$.
Every geometric fibre of $p:U\to\mathcal M$ is either $\PP^1$ or two
projective lines meeting at one node, with $L$-degree one on each
component in the latter case.
Indeed, ampleness and total degree two allow at most two noncontracted
components. A contracted component in an unmarked stable genus-zero
map has valency at least three, forcing at least three noncontracted
leaves in the dual tree. Thus none occurs. The smooth-domain locus is
dense, since it is open and contains $[f_0]$.

Let $I_z$ be the reduced irreducible component of $F^{-1}(z)$ containing
$B_z$. It is a proper integral Deligne--Mumford stack, and $B_z$ is
a dense open subset. Construct the proper Deligne--Mumford stack
\[
\mathcal C_{I_z}=I_z\times_{\mathcal M}U.
\]
It records $(f,a,b)$ on the original unmarked source: $b$ may equal $a$
or be a node. In particular, no contracted component is inserted when
the two points coincide. This family is flat over the integral stack
$I_z$, with geometrically reduced fibres and a geometrically integral
generic fibre, so $\mathcal C_{I_z}$ is integral.
By the Keel--Mori theorem \cite{KM}, $\mathcal C_{I_z}$ admits a coarse
moduli map $\pi:\mathcal C_{I_z}\to Q_z$. Let $W_z\to Q_z$ be the
normalization. At this stage $Q_z$ is a proper
integral algebraic space, and $W_z$ is a normal proper integral
algebraic space finite over $Q_z$. Trivial stabilizers on
$\mathcal U_z$ and its smoothness identify it with a dense open subset
of both $Q_z$ and $W_z$.

The induced morphism
\[
Q_z\longrightarrow\overline M\times Y,\qquad
(f,a,b)\longmapsto([f],f(b)),
\]
is proper with finite geometric fibres: over $([f],y)$, its points are
represented by pairs in $f^{-1}(z)\times f^{-1}(y)$, modulo automorphisms.
These sets are finite because $f$ contracts no component. The morphism
is therefore finite. A finite morphism of algebraic spaces is
representable by schemes \cite[Tag~03ZQ]{Stacks}, so $Q_z$ is a scheme.
Since finite morphisms of schemes are projective \cite[Tag~0B3I]{Stacks}
and $\overline M\times Y$ is projective, $Q_z$ is a projective variety.
Its finite normalization $W_z$ is a normal
projective variety as well.
Second evaluation induces $e_z:W_z\to Y$.

\smallskip
\noindent \textit{Extension and descent.}
On a finite-type reduced \'etale scheme atlas of $\mathcal M$, the
universal domain is a scheme, proper, flat and nodal over the atlas,
with smooth fibres over a dense open subset of every component.
Thus Lemma~\ref{cf:cancellation}
applies to $\mathcal L=F^*L$ and the section
$\vartheta\in H^0(U,\Omega^1_{U/\mathcal M}\otimes F^*L)$.
Here we view $\theta$ as a section of $\Omega_Y^1\otimes L$,
pull it back by $F$, and apply the relative differential
$F^*\Omega_Y^1\to\Omega^1_{U/\mathcal M}$ to obtain $\vartheta$.
The lemma gives a regular section
\[
\widetilde V\in H^0\bigl(U\times_{\mathcal M}U,
\mathrm{pr}_a^*F^*L\otimes\mathrm{pr}_b^*\omega_{U/\mathcal M}^{-1}\bigr).
\]
Uniqueness makes the construction compatible with isomorphisms of
stable maps, so the sections descend from the atlas. Restrict to
$\mathcal C_{I_z}$ and choose a nonzero element of $L_z$ to trivialize
the first factor. This gives
\[
V_z\in H^0(\mathcal C_{I_z},\omega_{\mathcal C_{I_z}/I_z}^{-1}).
\]
For second evaluation $e_b:\mathcal C_{I_z}\to Y$, the composite
\[
e_b^*\Omega_Y^1\longrightarrow\Omega^1_{\mathcal C_{I_z}/I_z}
\xrightarrow{\rho}\omega_{\mathcal C_{I_z}/I_z}
\xrightarrow{V_z}\OO_{\mathcal C_{I_z}}
\]
defines a section of $e_b^*TY$. Write $e_b=\bar e_b\circ\pi$, where
$\bar e_b:Q_z\to Y$. The bundle $e_b^*TY=\pi^*\bar e_b^*TY$ is
pulled back from the coarse space, so stabilizers act trivially on
its fibres. The invariant-function property of the geometric quotient
\cite[Theorem~1.1 and Definition~1.8(F)]{KM} gives
$\pi_*\OO_{\mathcal C_{I_z}}=\OO_{Q_z}$.
On open subsets of $Q_z$ trivializing $\bar e_b^*TY$, this identity
applies in each coordinate, yielding the projection formula
$\pi_*(e_b^*TY)\simeq\bar e_b^*TY$. The section therefore descends to
$Q_z$; pulling back to $W_z$ gives $S_z$.
On the source of $(f,a)\in B_z$, choose $u(a)=0$. The two-point formula
gives $V_{f,a}=c\,u^2\partial_u$ with $c\ne0$, since $(f^*\theta)\circ df$ is
an isomorphism. Hence $S_z(f,a,b)=df_b(V_{f,a}(b))$, which is nonzero
for $b\ne a$ because $f$ is unramified.

\smallskip
\noindent \textit{Evaluation.}
Both $\mathcal U_z$ and $Y$ are smooth of dimension $2n+1$.
For $b\ne a$, applying \eqref{cf:pointed-moduli-tangent} to two-point
evaluation gives
\[
\ker(de_z)_{(f,a,b)}
\simeq
\frac{H^0(f^*TY(-[a]-[b]))}{df\,H^0(T\PP^1(-[a]-[b]))}=0.
\]
Indeed, $f^*TY(-[a]-[b])\simeq\OO\oplus\OO(-1)^{\oplus2n}$, and $df$
identifies its constant summand with $T\PP^1(-[a]-[b])\simeq\OO$.
Thus $(de_z)_{(f,a,b)}$ is injective, hence an isomorphism between
tangent spaces of smooth varieties of the same dimension. It follows
that $e_z$ is \'etale there.

Since $W_z$ is projective, the image of $e_z$ is closed. It contains
a nonempty open subset by \'etaleness, so it equals the irreducible
variety $Y$. The equality $\dim W_z=\dim Y$ then gives generic
finiteness.
\end{proof}

\subsection{The VMRT and contact homogeneity}\label{cf:sec-homogeneity}

Keep the assumptions on $Y$ from Subsection~\ref{cf:sec-incidence},
and assume in addition that $\Pic(Y)=\ZZ[L]$.

The incidence sections give local vector fields whose values span
each tangent space, as follows.

\begin{lemma}\label{cf:span}
For every $y\in Y$, there are balanced transverse maps
$f_1,\ldots,f_{2n+1}$, with source points
$b_i$ mapping to $y$, whose tangent lines span $T_yY$.
For any $a_i\ne b_i$, the construction of Theorem~\ref{cf:incidence},
based at $(f_i,a_i)$, gives a pointed \'etale neighbourhood
$(E_i,q_i)\to(Y,y)$, with $q_i=(f_i,a_i,b_i)$, and a vector field
$v_i$ on $E_i$ whose value, under $T_{q_i}E_i\simeq T_yY$, is nonzero
and lies in $df_{i,b_i}(T_{b_i}\PP^1)$. These values span $T_yY$.
\end{lemma}

\begin{proof}
Choose the maps $f_i$ with $b_i=0$ by Corollary~\ref{cf:every-point}.
Put $z_i=f_i(a_i)$ and apply Theorem~\ref{cf:incidence} to obtain
$e_i:W_i\to Y$ and $S_i\in H^0(W_i,e_i^*TY)$. By
\textup{(\ref{cf:incidence-evaluation})}, $e_i$ is \'etale on an open
neighbourhood $E_i\subset W_i$ of $q_i$. Hence
$v_i=(de_i)^{-1}(S_i|_{E_i})$ is a vector field on $E_i$, and
\textup{(\ref{cf:incidence-tangency})} gives
\[
de_i\bigl(v_i(q_i)\bigr)=S_i(q_i)
\in df_{i,b_i}(T_{b_i}\PP^1)\setminus\{0\}.
\]
The chosen tangent lines span $T_yY$, so these values do as well.
\end{proof}

Recall from Subsection~\ref{cf:sec-lines} the chosen contact-line
component $\mathcal K$ and its VMRT $Z_y$. Apply the construction of
Subsection~\ref{cf:sec-general-tools} to this $\mathcal K$, and write
$p:\mathcal V\to Y$ for the resulting integral projective
tangent-direction space. Theorem~\ref{cf:vmrt-facts} supplies the
general-point properties of $Z_y$ used below. We will first identify
the general scheme-theoretic fibre of $p$ with $Z_y$, then prove
smoothness of $p$ and irreducibility of its fibres.

\begin{proposition}\label{cf:nonlinear}
If every point of $Y$ lies on a balanced transverse map,
then $p:\mathcal V\to Y$ is smooth with irreducible fibres.
For general $y$ its scheme-theoretic fibre is $Z_y$, which
is positive-dimensional and nonlinear.
\end{proposition}

\begin{proof}
\emph{The general fibre.}
The generic fibre of $p$ is integral, hence geometrically reduced in
characteristic zero. Generic flatness and openness of geometric
reducedness give a nonempty open $U\subset Y$ over which $p$ is flat
with reduced fibres. By Theorem~\ref{cf:vmrt-facts}\textup{(\ref{cf:vmrt-lines})},
shrink $U$ so that all members of $\mathcal K$ meeting $U$ are smooth.
The tangent map from its normalized universal
family over $U$ is then a proper morphism with reduced image
$\mathcal V|_U$. Thus the fibres of $p$ over $U$ are precisely the
VMRTs $Z_y$, including their reduced scheme structures. By
Theorem~\ref{cf:vmrt-facts}\textup{(\ref{cf:vmrt-geometry})}, after
shrinking $U$ these fibres are smooth, so $p$ is smooth over $U$.

\smallskip
\noindent \emph{Smoothness at every point.}
At any $y\in Y$, Lemma~\ref{cf:span} supplies \'etale local vector
fields with spanning values. Proposition~\ref{cf:jets} lifts each
field to all of the corresponding base change of $\mathcal V$.
Lemma~\ref{cf:smoothness} therefore makes $p$ smooth everywhere.

\smallskip
\noindent \emph{Irreducibility and nonlinearity.}
In the Stein factorization $\mathcal V\to Y'\xrightarrow{q}Y$, the finite
map $q$ is \'etale because $p$ is smooth and projective
\cite[Lemma~2.4]{FujinoGongyo}. The Fano condition forces this connected
finite \'etale cover to be trivial, as the following calculation shows.
The variety $Y'$ is smooth projective Fano, with
$K_{Y'}=q^*K_Y$ and $TY'=q^*TY$.
Kodaira vanishing \cite{Kodaira} gives
$\chi(\OO_Y)=\chi(\OO_{Y'})=1$, while
Hirzebruch--Riemann--Roch \cite{HRR} yields
\[
\chi(\OO_{Y'})=\deg(q)\,\chi(\OO_Y).
\]
Hence $\deg(q)=1$, so the smooth fibres of $p$ are
connected, hence irreducible. Finally,
Theorem~\ref{cf:vmrt-facts}\textup{(\ref{cf:vmrt-geometry})--(\ref{cf:vmrt-span})}
gives $\dim Z_y=n-1>0$, whereas its span $\PP(D_y)$ has dimension
$2n-1$. Thus the
irreducible general VMRT is not linear.
\end{proof}

The required descent result allows singular projective sources, as
needed for the incidence compactifications.

\begin{theorem}[{Hwang, \cite[Theorem~1.3]{HwangDescent}}]
\label{cf:descent}
Let $X$ be a smooth complex Fano variety of Picard number one with
positive-dimensional nonlinear general VMRT for a minimal rational
component. For every projective variety $W$ and surjective morphism
$e:W\to X$, pullback induces an isomorphism
\[
H^0(X,TX)\xrightarrow{\sim}H^0(W,e^*TX).
\]
Here nonlinear means that some irreducible component of the general
VMRT is not a projective linear subspace.
\end{theorem}

The following consequence gives a general homogeneity criterion,
without assuming a contact structure.

\begin{proposition}[A criterion for homogeneity]\label{cf:generation}
Let $X$ be a connected smooth complex projective Fano variety of
Picard number one with positive-dimensional nonlinear general VMRT
for a minimal rational component. Fix $x\in X$. Suppose there are
finitely many integral projective varieties $W_i$, surjective morphisms
$e_i:W_i\to X$, sections $S_i\in H^0(W_i,e_i^*TX)$, and points
$w_i\in e_i^{-1}(x)$ such that the values
\[
S_i(w_i)\in(e_i^*TX)_{w_i}=T_xX
\]
span $T_xX$. Then $H^0(X,TX)\to T_xX$ is surjective, and the
$\Aut(X)^0$-orbit through $x$ is open. If this condition holds at
every $x\in X$, then $TX$ is globally generated and $\Aut(X)^0$
acts transitively on $X$.
\end{proposition}

\begin{proof}
By Theorem~\ref{cf:descent}, each $S_i$ is the pullback of a global
vector field $H_i\in H^0(X,TX)$. Thus
\[
H_i(x)=S_i(w_i),
\]
so evaluation at $x$ is surjective. Put $G=\Aut(X)^0$. The standard
identification $\operatorname{Lie}G=H^0(X,TX)$ identifies evaluation
at $x$ with the differential at the identity of the orbit map
$G\to X$, $g\mapsto g(x)$
\cite[Theorem~2.1 and Corollary~2.2]{BrionSpherical}.
Surjectivity makes this map smooth at the identity, so its image
contains an open neighbourhood of $x$. The $G$-translates of this
neighbourhood cover the orbit $G\cdot x$, proving that it is open.
If the condition holds at every point, evaluation is everywhere surjective,
so $TX$ is globally generated and every orbit is open. Connectedness
of $X$ then leaves only one orbit.
\end{proof}

The same descent gives embeddedness before the homogeneity conclusion.
For any chosen source point, the incidence section vanishes there and
nowhere else on that source. A section pulled back from $Y$ cannot
have this property if another source point has the same image.

\begin{proposition}[Embeddedness of balanced transverse maps]\label{cf:bt-embedding}
Let $(Y,D)$ be a smooth complex projective contact Fano manifold of
dimension $2n+1$, with $n\ge2$ and $\Pic(Y)=\ZZ[L]$.
Every balanced transverse map $f:\PP^1\to Y$ is a closed
immersion, with $N_{f(\PP^1)/Y}\simeq\OO(1)^{\oplus2n}$.
\end{proposition}

\begin{proof}
For the minimal rational component $\mathcal K$ fixed in
Subsection~\ref{cf:sec-lines}, Corollary~\ref{cf:every-point} and
Proposition~\ref{cf:nonlinear} give a positive-dimensional nonlinear
general VMRT. Fix any such map $f$ and any $a\in\PP^1$, and put $z=f(a)$.
Theorem~\ref{cf:incidence}, applied with $(f_0,a_0)=(f,a)$, supplies a
normal projective variety $W_z$, a surjective morphism $e_z:W_z\to Y$,
and the section $S_z$. Since $Y$ is Fano of Picard number one,
Theorem~\ref{cf:descent} applies and gives a global vector field
$H_a\in H^0(Y,TY)$ with $S_z=e_z^*H_a$.

The whole source of $f$, including $a$, is a fibre of
$\mathcal U_z\to B_z$ contained in $W_z$.
By Theorem~\ref{cf:incidence}\textup{(\ref{cf:incidence-tangency})},
restriction to this fibre gives
\[
H_a(f(a))=0,\qquad
H_a(f(b))=df_b\bigl(V_{f,a}(b)\bigr)\ne0\quad(b\ne a).
\]
Thus $f(b)\ne f(a)$ for every $b\ne a$. Since $a$ was arbitrary,
$f$ is injective.

The morphism $f$ is proper and quasi-finite, hence finite. Since it
is unramified and injective, every nonempty closed fibre is a single
reduced point. Nakayama's lemma makes
$\OO_Y\to f_*\OO_{\PP^1}$ surjective, so $f$ is a closed immersion.
The balanced splitting gives the asserted normal bundle.
\end{proof}

\begin{proof}[Proof of Theorem~\ref{cf:main}]
\phantomsection\label{cf:proof-main}
In dimension one, $Y\simeq\PP^1$, $D=0$, and $L\simeq\OO_{\PP^1}(2)$,
so the assertions are immediate. In higher dimensions,
Lemma~\ref{cf:reduction} gives contact homogeneity for the two classical models.
In the remaining case, Proposition~\ref{cf:first-conic} constructs one
very free map of contact degree two, and Corollary~\ref{cf:every-point}
gives balanced transverse maps through every point.
Proposition~\ref{cf:nonlinear} then supplies a positive-dimensional
nonlinear general VMRT.

For each $y\in Y$, Lemma~\ref{cf:span} and
Theorem~\ref{cf:incidence} provide normal projective varieties $W_i$,
surjective maps $e_i:W_i\to Y$, sections
$S_i\in H^0(W_i,e_i^*TY)$, and points $q_i\in e_i^{-1}(y)$ whose
values $S_i(q_i)$ span $T_yY$. Proposition~\ref{cf:generation}
therefore makes $TY$ globally generated and the action of
$\Aut(Y)^0$ transitive.

Under $\Pic(Y)=\ZZ[L]$, the contact structure is unique
\cite[Corollary~4.5]{Kebekus}; the projective-space and projectivized
cotangent-bundle exceptions do not satisfy this hypothesis.
Thus $\Aut(Y)^0$ preserves $D$, giving the required transitive contact
action. On the two classical models, global generation of $TY$
also follows from their transitive contact actions.

In all cases, $G=\Aut(Y,D)^0$ acts transitively on $Y$.
Its action on a high anticanonical embedding identifies $G$ with a
closed subgroup of a projective linear group, so $G$ is a connected
affine algebraic group. By Borel's fixed-point theorem, a Borel
subgroup fixes a point. The connected solvable radical lies in this
Borel and, by normality and transitivity, fixes every point of $Y$.
Effectiveness makes it trivial. Thus $G$ is semisimple and $Y=G/P$
for a parabolic subgroup $P$. Every ample line bundle on $G/P$ is
very ample \cite[Exercise~3.1.E(1)]{BrionKumar}, applied after passing
to the simply connected cover of $G$. Hence the ample contact line
bundle $L$ is very ample.
\end{proof}

The very ampleness of $L$ now gives a second algebraic embeddedness
argument, which also applies to maps that are not transverse.

\begin{proposition}\label{cf:degree-two-images}
Let $(Y,D)$ be a smooth complex projective contact Fano manifold of
dimension $2n+1$, with contact line bundle $L$.
For every morphism $f:\PP^1\to Y$ with $\deg f^*L=2$, the reduced
image $C=f(\PP^1)_{\mathrm{red}}$ is a smooth rational curve,
identified by the $L$-embedding with a line or a smooth plane conic.
Consequently every unramified such $f$ is a closed immersion.
In particular, every balanced map is a closed immersion, with
$N_{C/Y}\simeq\OO(1)^{\oplus2n}$.
\end{proposition}

\begin{proof}
By Theorem~\ref{cf:main}, $L$ is very ample. Let
$\iota:Y\hookrightarrow\PP(H^0(Y,L)^*)$ be its associated embedding.

The coordinate sections of $\iota\circ f$ lie in
$H^0(\PP^1,\OO(2))$, so its image spans at most a plane. Moreover,
\[
2=\deg f^*L=\deg(f:\PP^1\to C)\,\deg_L C.
\]
Thus $\iota(C)$ is a line or an integral plane conic. Over $\CC$,
both are smooth, so $C$ is normal. In the conic case, $f$ is a finite
birational morphism onto a normal curve, hence an isomorphism onto
$C$. The remaining case is a double cover of a line, which is
ramified by Riemann--Hurwitz. Hence an unramified $f$ is a closed
immersion.

Every balanced map has contact degree two and is unramified, with
$N_f\simeq\OO(1)^{\oplus2n}$, as noted after
Definition~\ref{cf:balanced-definition}. Applying the preceding
assertion proves the final claim.
\end{proof}

\begin{proof}[Proof of Corollary~\ref{cf:wolf}]
\phantomsection\label{cf:proof-wolf}
The twistor space $(Y,D)$ of $(X,g)$ is a smooth projective contact
Fano manifold \cite[Theorem~2.1]{LeBrun}. By
Theorem~\ref{cf:main}, it is homogeneous under contact automorphisms.
The classification of homogeneous contact manifolds
\cite{Wolf,Beauville} identifies it with the twistor space of a Wolf
space. LeBrun's uniqueness theorem
\cite[Corollaries~3.3--3.4]{LeBrun} then identifies the original
quaternionic--K\"ahler metric with the Wolf-space metric up to a
positive constant factor.
\end{proof}

\begin{proof}[Proof of Corollary~\ref{cf:ke-polystability}]
\phantomsection\label{cf:proof-ke-polystability}
As shown in the proof of Theorem~\ref{cf:main}, $Y=G/P$ is a rational
homogeneous variety. It therefore admits an invariant K\"ahler--Einstein
metric of positive Ricci curvature
\cite[\S5, p.~172]{Matsushima72}. K-polystability then follows from
\cite[Theorem~1.1]{Berman16}.
\end{proof}

\begin{proof}[Proof of Corollary~\ref{cf:embedded-bt}]
\phantomsection\label{cf:proof-embedded-bt}
Proposition~\ref{cf:degree-two-images} gives embeddedness and the
normal bundle for every balanced map. In the primitive Picard case,
Corollary~\ref{cf:every-point} and Proposition~\ref{cf:bt-embedding}
already give transverse quasi-lines through every point by incidence
and descent.

For $n=0$, take $C=Y\simeq\PP^1$. For $n\ge1$,
Corollary~\ref{cf:ke-polystability} and LeBrun's correspondence
\cite[Theorem~A]{LeBrun} identify $(Y,D)$ with a twistor space.
By \cite[Theorem~2.1(ii)--(iii)]{LeBrun}, its fibres are embedded
copies of $\PP^1$ with normal bundle $\OO(1)^{\oplus2n}$, and
$D$ is the horizontal distribution. Hence each fibre $C$ satisfies
$TY|_C=TC\oplus D|_C$ and
$\theta|_{TC}:TC\xrightarrow{\sim}L|_C$, so
\[
TY|_C\simeq\OO_{\PP^1}(2)\oplus\OO_{\PP^1}(1)^{\oplus2n}.
\]
This proves the twistor assertion and also supplies the transverse
curves through every point in the classical cases.
\end{proof}

\section*{Acknowledgments}

The first named author is supported by Overseas High-Level Young Talent Recruitment
Programs and is partially supported by the Chongqing Natural Science Foundation Innovation
and Development Joint Fund CSTB2023NSCQ-LZX0031. The second named author is partially
supported by the National Key Research and Development Program of China (Grant No.
2024YFA1013201), the Natural Science Foundation of Chongqing (Grant No.
CSTB2024NSCQ-LZX0040), and the Special Project of Chongqing Municipal Science and Technology
Bureau (Grant No.\ 2025CCZ015). The third named author is sponsored by the National Key R\&D
Program of China (Grant No.\ 2024YFA1013200) and the National Natural Science Foundation of
China (Grant No.\ 12671100).

\end{document}